\documentclass[envcountsame,envcountsect]{svmult}
\usepackage{amssymb,amsmath,amsrefs,newlfont,enumerate,array,graphicx,tikz,tkz-tab,tikz-cd,xcolor,float,subcaption,graphicx}
\usepackage{hyperref, url, cite, microtype}

\allowdisplaybreaks
\usepackage{mathtools}
\mathtoolsset{showonlyrefs}

\makeatletter
\let\c@equation\c@theorem
\makeatother

\makeatletter

\spnewtheorem{example}[theorem]{Example}{\bfseries}{\rmfamily}
\makeatother

\usepackage{palatino}
\usepackage[font=small,labelfont=bf]{caption}
\newcommand{\C}{\mathbb{C}}
\newcommand{\conv}{\operatorname{conv}}
\newcommand{\DS}{\mathsf{DS}}
\newcommand{\DSS}{\mathsf{DSS}}
\newcommand{\e}{\mathrm{e}}
\newcommand{\ex}{\operatorname{ex}}
\newcommand{\F}{\mathrm{F}}
\newcommand{\h}{\mathsf{H}}
\newcommand{\M}{\mathsf{M}}
\newcommand{\Perm}{\mathsf{P}}
\newcommand{\PermSub}{\mathsf{PS}}
\newcommand{\R}{\mathbb{R}}
\newcommand{\T}{\mathsf{T}}
\newcommand{\U}{\mathcal{U}}
\newcommand{\V}{\mathcal{V}}
\newcommand{\W}{\mathcal{W}}

\newcommand{\rank}{\operatorname{rank}}
\newcommand{\tr}{\operatorname{tr}}
\newcommand{\Span}{\operatorname{span}}
\newcommand{\w}{\mathrm{w}}
\newcommand{\0}{{\color{lightgray}0}}

\renewcommand{\leq}{\leqslant}
\renewcommand{\geq}{\geqslant}
\renewcommand{\subset}{\subseteq}
\renewcommand{\phi}{\varphi}
\renewcommand{\vec}[1]{\mathbf{#1}}

\newcommand{\diag}{\operatorname{diag}}

\begin{document}

\title*{Unitarily invariant norms}
\titlerunning{Unitarily invariant norms}

\author{Stephan Ramon Garcia\, \orcidID{0000-0001-8971-5448}\\ Javad Mashreghi\, \orcidID{0000-0002-7969-9576}\\ Marek Ptak\,\orcidID{0000-0002-3843-7932}\\ William T. Ross\,\orcidID{0000-0002-3357-3767}}

\institute{Stephan Ramon Garcia \at Department of Mathematics and Statistics, Pomona College, Claremont, California, 91711, USA, \email{stephan.garcia@pomona.edu}, \url{https://stephangarcia.sites.pomona.edu/} \and Javad Mashreghi \at D\'epartement de math\'ematiques et de statistique, Universit\'e Laval, Qu\'ebec, QC, Canada, G1K 0A6, \email{javad.mashreghi@mat.ulaval.ca} \and Marek Ptak \at Department of Applied Mathematics, University of Agriculture, ul. Balicka 253c  30-198 Krak\'ow, Poland, \email{rmptak@cyf-kr.edu.pl} \and William T. Ross \at Department of Mathematics and Computer Science, University of Richmond, Richmond, VA 23173, USA, \email{wross@richmond.edu}}




\maketitle
\renewcommand{\theequation}{\thetheorem}    

\vspace{-1in}
\medskip

\abstract{
This survey paper provides a comprehensive study of unitarily invariant norms on the algebra of $n \times n$ matrices. This investigation leads naturally to the theory of symmetric gauge functions, a class of norms on $\R^n$ characterized by invariance and monotonicity properties. We develop the necessary framework by examining absolute and monotone norms and establishing their equivalence, thereby offering additional insight into the classical Hardy--Littlewood--P\'{o}lya theorem on majorization. The theory of majorization is further explored through its connections with doubly stochastic matrices, convexity, and fundamental results such as the Birkhoff and Rad\'o theorems, as well as K\"{o}nig’s theorem on term rank and line rank. We also study weak majorization and derive a characterization that plays a crucial role in proving the monotonicity of symmetric gauge functions. On the spectral side, we review key results in matrix analysis, including the Courant--Fischer min–max theorem, the Cauchy interlacing theorem, and Ky Fan’s majorization theorem, along with a weak subadditivity result for singular values of arbitrary matrices. These ingredients culminate in a detailed proof of von Neumann’s characterization of unitarily invariant norms, which provides a complete and elegant description of this class of norms. Some illustrative examples, as well as the Ky Fan domination principle as an application, are also presented.}


\section{Introduction}
Let $\M_{m \times n}$ denote the set of all $m \times n$ complex matrices. For convenience, we let $\M_n = \M_{n \times n}$. Since $\M_n$ is a finite-dimensional vector space, all norms on it are equivalent: if $\|\cdot\|_{\mathrm{a}}$ and $\|\cdot\|_{\mathrm{b}}$ are norms on $\M_n$, then there is a $c > 0$ such that 
$c^{-1} \|A\|_{\mathrm{a}} \leq \|A\|_{\mathrm{b}} \leq c \|A\|_{\mathrm{a}}$ for all $A \in \M_{n}$.  Consequently,
every norm on $\M_n$ gives rise to the same open sets, closed sets, and convergent sequences.
Despite this, some norms possess additional features that make them more interesting. In this survey, we study the \emph{unitarily invariant norms} on $\M_n$. These are the norms  $\|\cdot\|$ on $\M_n$ satisfying
\begin{equation}\label{E:UIN-def-norm}
\|UAV\| = \|A\|
\end{equation}
for all $A \in \M_n$ and all unitary $U, V \in \M_n$. Recall that $U \in \M_n$ is \emph{unitary} if $U^{*} U = I$;
that is, if its columns are orthonormal in $\C^n$.

For each $A \in \M_n$, the matrix $A^{*} A$ is positive semidefinite and hence has nonnegative eigenvalues. The nonnegative square roots of these eigenvalues are the \emph{singular values} of $A$. It is traditional to denote them as
\begin{equation}\label{eq:SingularDecrease}
s_1(A) \geq s_2(A) \geq \cdots \geq s_n(A) \geq 0.
\end{equation}
Write
\begin{equation*}
\vec{s}(A) :=
\begin{bmatrix}
s_1(A) \\
s_2(A) \\
\vdots\\
s_n(A)
\end{bmatrix} \in \R^n_{+}\!\!\downarrow,
\end{equation*}
in which $\R^n_{+}\!\!\downarrow$ is the subset of $\R^n$ consisting of all 
$\vec{x} = [x_i] \in \R^n$ such that $x_1 \geq x_2 \geq \cdots \geq x_n \geq 0$; see Figure \ref{fig:region-R2d}.  Here $\R_+ = [0,\infty)$.

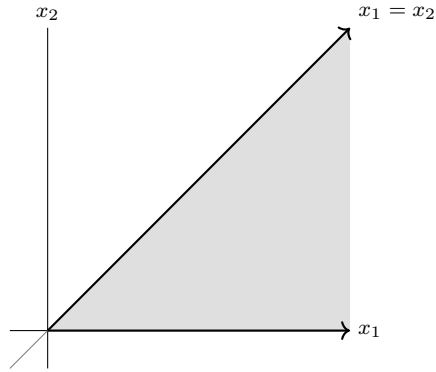
\begin{figure}
\centering
\begin{tikzpicture}[scale=1.0]
%
\draw[] (-0.5,0) -- (4,0) node[right] {$x_1$};
\draw[] (0,-0.5) -- (0,4) node[above] {$x_2$};
%
\fill[gray!25] (0,0) -- (4,0) -- (4,4) -- cycle;
 Line x1 = x2
\draw[thick,->] (0,0) -- (4,4) node[above right] {$x_1=x_2$};
\draw[gray] (-0.5,-0.5) -- (0,0);
\draw[thick,->] (0,0) -- (4,0);
\end{tikzpicture}
\caption{The region $\R^2_{+}\!\!\downarrow$.}
\label{fig:region-R2d}
\end{figure}

The singular value decomposition provides unitaries $U,V\in \M_n$ such that
\begin{equation}\label{E:UIN-A=udiagV-0}
A = U (\diag \vec{s}(A))  V,  
\end{equation}
in which $\diag \vec{x}$ denotes the $n \times n$ diagonal matrix whose diagonal entries are the entries of $\vec{x} \in \R^n$, in the same order  \cite[Theorem 1.12]{MR3076701}, \cite[Theorem 16.1.5]{SCLA2}. Consequently, if $\|\cdot\|$ is a unitarily invariant norm on $\M_n$, then
\begin{equation}\label{E:UIN-A=udiagV}
\|A\| = \|\diag (s(A))\|.
\end{equation}
Thus, $\|A\|$ is determined by its vector of singular values $\vec{s}(A) \in \R^n_{+}\!\!\downarrow$. 

Given $\phi: \R^n_{+}\!\!\downarrow  \to \R_{+}$, define
\begin{equation}\label{eq:psA}
\|A\| := \phi( \vec{s}(A) )
\end{equation}
for $A \in \M_n$.  In this survey, we provide necessary and sufficient conditions on $\phi$ such that \eqref{eq:psA} is a unitarily invariant norm on $\M_n$. This inquiry leads to symmetric gauge functions, a class of norms on $\R^n$ endowed with certain structural properties. In the following sections, we develop tools to study these objects and we then use them in \S \ref{S:UIN-on-neumann} to establish von Neumann’s characterization of unitarily invariant norms. 

Although von Neumann’s proof is both elegant and relatively short, a full understanding of its components requires familiarity with several important topics, including the min–max theorem for Hermitian matrices, majorization, convexity, and symmetric gauge functions. Figure~\ref{fig:theorems-needed} summarizes the ten principal results involved. Broadly speaking, the left column pertains to the spectral theory of Hermitian matrices, while the other two columns focus on majorization and its connections with doubly stochastic matrices and convexity. These two strands ultimately come together to yield von Neumann’s theorem.

The structure of the paper is as follows. In \S \ref{S:UIN-examples}, we begin with two fundamental examples of unitarily invariant norms: the spectral norm and the Frobenius norm. In \S \ref{S:UIN-absolute-norms}, we study absolute norms and monotone norms, and establish their equivalence via an induction argument which also helps clarify the proof of the Hardy--Littlewood--P\'{o}lya theorem presented in \S \ref{S:UIN-majoraization}. The notions of term rank and line rank are introduced in \S \ref{S:UIN-konig}, in which K\"{o}nig’s theorem on their equivalence is proved. In \S \ref{S:UIN-birkhoff}, we investigate convex bodies and prove the Birkhoff and Rad\'o theorems. \S \ref{S:UIN-weak-major} is devoted to weak majorization, for which we provide a characterization based on majorization results developed in \S\S \ref{S:UIN-majoraization} and \ref{S:UIN-birkhoff}. We use these results, together with Rad\'o’s theorem from \S \ref{S:UIN-birkhoff}, to establish the monotonicity of symmetric gauge functions in \S \ref{S:UIN-monotonicity}. This latter result is one of the key ingredients in the proof of von Neumann’s theorem.

To develop the material on the left side of Figure \ref{fig:theorems-needed}, we begin with deeper properties of Hermitian matrices, such as the Courant--Fischer theorem in \S \ref{S:UIN-the-min-max-thm}. In \S \ref{UIN-Cauchy}, we study Cauchy's interlacing theorem (Theorem \ref{T:UIN-cauchy}). Section  \ref{S:UIN-fan} presents Ky Fan’s majorization theorem (Theorem \ref{T:UIN-fan}). In \S \ref{S:UIN-weak-aubadd}, we establish an analogous result for singular values of arbitrary matrices in $\M_n$.  While Ky Fan’s theorem shows that the eigenvalues of Hermitian matrices are weakly subadditive, Theorem \ref{T:UIN-weak-aubadd} demonstrates that the singular values of general matrices are weakly subadditive. Finally, in \S \ref{S:UIN-on-neumann}, we present von Neumann’s characterization of unitarily invariant norms. In the final two sections, \S\S\ref{S:UIN-further-examples}--\ref{S:UIN-further-applications}, we provide additional examples of unitarily invariant norms along with the Ky Fan domination principle as an application of von Neumann's characterization.

\bigskip\noindent\textbf{Notation.}
Our inner products are linear in the first slot and our vectors are column vectors, so $\vec{x} = [x_i] \in \C^n$ denotes the $n \times 1$ matrix with entries $x_1,x_2,\ldots,x_n$.  On rare occasions, when typographical necessity demands it, we may write a vector inline as $(x_1,x_2,\ldots,x_n)$, or even $[x_1~x_2~\ldots~x_n]^{\T}$, in which the superscript ${}^{\T}$ denotes the transpose.  The all-ones vector is denoted by $\vec{1}$, in which the length can be inferred by context. We let $\M_{m \times n}$ denote the set of $m \times n$ complex matrices; $\M_n = \M_{n \times n}$ for short. The set of $n \times n$ Hermitian matrices is $\h_n$.  If $A,B\in \h_n$, then $A \geq B$ means that $A - B$ is positive semidefinite.  Below $\|\cdot\|$ denotes the Euclidean norm on $\C^n$.  We adorn other norms with appropriate subscripts. We use $:=$ when defining something globally; that is, when defining something used throughout the remainder of the paper.

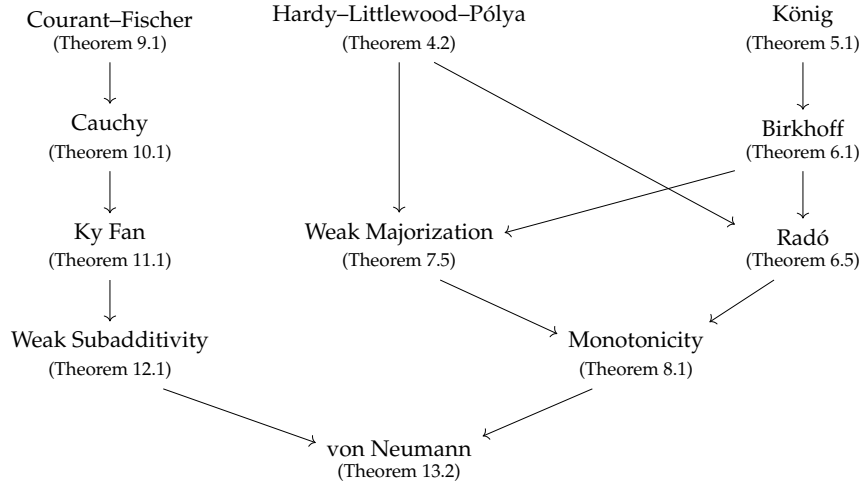
\begin{figure}
\newcommand{\thmnode}[2]{%
  \shortstack[c]{#1\\ {\scriptsize (Theorem \ref{#2})}}%
}
\centering
\begin{tikzcd}[row sep=16pt, column sep=8pt]
\thmnode{Courant--Fischer}{T:max-min}  \arrow[d] &  &  \thmnode{Hardy--Littlewood--P\'{o}lya}{T:UIN-Hardy-Littlewood-Polya} \arrow[dd]\arrow[rrdd]   &  & \thmnode{K\"{o}nig}{T:UIN-konig} \arrow[d]                \\
\thmnode{Cauchy}{T:UIN-cauchy} \arrow[d]                                        &  &   &  & \thmnode{Birkhoff}{T:UIN-birkhoff} \arrow[d] \arrow[lld] \\
\thmnode{Ky Fan}{T:UIN-fan} \arrow[d]                                           &  & \thmnode{Weak Majorization}{T:UIN-weak-major-char-2} \arrow[rd] &  & \thmnode{Rad\'o}{T:UIN-Rado} \arrow[ld]                 \\
\thmnode{Weak Subadditivity}{T:UIN-weak-aubadd} \arrow[rrd]                                      &  &                  & \thmnode{Monotonicity}{T:UIN-monotonicity} \arrow[ld] &            \\                                                   &  & \thmnode{von Neumann}{T:UIN-von-Neumann}             &  &
\end{tikzcd}
\caption{Ingredients for von Neumann's theorem.}
\label{fig:theorems-needed}
\end{figure}

\section{Two fundamental examples} \label{S:UIN-examples}
We start with two well-known norms on $\M_n$ that satisfy \eqref{E:UIN-def-norm}: the spectral norm and the Frobenius norm. These important examples help illustrate a number of standard techniques that are useful in what follows.

\subsection*{The spectral norm} 
The \emph{spectral norm} of $A \in \M_n$ is defined by
\begin{equation*}
\|A\|_{\mathrm{op}} := \sup_{\|\vec{x}\| = 1} \|A\vec{x}\|,
\end{equation*}
in which $\vec{x}$ runs over the unit sphere of $\C^n$.  
The notation stems from the fact that the spectral norm is the operator norm when we regard $A$ as a bounded linear operator on $\C^n$, equipped with the Euclidean norm. 
The definition and a scaling argument ensure that
\begin{equation}\label{eq:SpectralAx}
    \| A \vec{x} \| \leq \|A \|_{\mathrm{op}}\  \| \vec{x} \|
\end{equation}
for all $\vec{x} \in \C^n$; the inequality being automatic for $\vec{x} = \vec{0}$.

Let $U,V \in \M_n$ be unitary and observe that
\begin{equation*}
\|(UAV)\vec{x}\| = \|U(AV\vec{x})\| = \|AV\vec{x}\|
\end{equation*}
since $U$ is isometric.  If $\vec{y} = V\vec{x}$, then $\|\vec{y}\|=\|\vec{x}\|$ since $V$ is isometric.
Moreover, $\vec{y}$ runs over the unit sphere of $\C^n$ as $\vec{x}$ does. Therefore, 
\begin{align*}
\|UAV\|_{\mathrm{op}} 
&= \sup_{\|\vec{x}\| = 1}  \|UAV\vec{x}\| \\
&= \sup_{\|\vec{x}\| = 1}  \|AV\vec{x}\| \\
&= \sup_{\|\vec{y}\| = 1}  \|A\vec{y}\| \\
& = \|A\|_{\mathrm{op}}
\end{align*}
for all $A \in \M_n$ and all unitary matrices $U,V \in \M_n$. That is, the spectral norm is unitarily invariant. 

According to \eqref{E:UIN-A=udiagV}, we also see that
\begin{align}
\|A\|_{\mathrm{op}} &= \| \diag s(A)\|_{\mathrm{op}} \notag\\
&= \sup_{\|\vec{x}\|=1} \|\diag (s(A))\vec{x}\|  \notag\\
&= \sup_{|x_1|^2+\cdots+|x_n|^2=1}\left( s_1(A)|x_1|^2+s_2(A)|x_2|^2+\cdots+s_n(A)|x_n|^2 \right)^{\frac{1}{2}} \notag\\
&= s_1(A) \label{E:UIN-spectral-norm-s1}
\end{align}
by \eqref{eq:SingularDecrease}.
Note that \eqref{eq:SpectralAx} ensures that
\begin{equation*}
\|AB\vec{x}\| \leq \|A\|_{\mathrm{op}} \  \|B\vec{x}\| \leq \|A\|_{\mathrm{op}} \  \|B\|_{\mathrm{op}}\  \|\vec{x}\|
\end{equation*}
for all $\vec{x} \in \C^n$.  Consequently, the spectral norm is \emph{submultiplicative}:
\begin{equation}\label{E:UIN-spectral-norm-s-mul}
\|AB\|_{\mathrm{op}} \leq \|A\|_{\mathrm{op}}\   \|B\|_{\mathrm{op}}.
\end{equation}
In the matrix-theory literature, a norm on $\M_n$ that is submultiplicative is called a \emph{matrix norm}.  Beware: although ``matrix norm'' sounds like it is simply a norm on a space of matrices, it entails this additional property.

\subsection*{The Frobenius norm} 
The \emph{Frobenius norm} of $A = [a_{ij}]\in \M_n$ is
\begin{equation}\label{E:UIN-def-norm-F}
\|A\|_{\F} := \bigg(\sum_{i,j=1}^{n} |a_{ij}|^2\bigg)^{\frac{1}{2}}.
\end{equation}
There are at least two ways to verify that $\|\cdot\|_{\F}$ is unitarily invariant. The first is based on the formula
\begin{equation}\label{E:UIN-def-norm-F2}
\|A\|_{\F}^2 = \tr (A^*A)
\end{equation}
and the cyclic invariance of the trace: $\tr (XY) =\tr (YX)$. For all $A \in \M_n$ and all unitary matrices $U,V \in \M_n$,
\begin{align*}
\|UAV\|_{\F}^2 
&= \tr ((UAV)^*(UAV)) \\
&= \tr ((V^*A^*U^*)(UAV)) \\
&= \tr (V^*A^*AV) \\
&= \tr (A^*AVV^*) \\
&= \tr (A^*A)\\
&= \|A\|_{\F}^2.
\end{align*}

For the second approach, observe that $\|X^{\T} \|_{\F} = \| X\|_{\F}$ for all $X \in \M_n$.
If $A = [ \vec{a}_1~\vec{a}_2~\ldots~\vec{a}_n] \in \M_n$, in which each $\vec{a}_i \in \C^n$, and $U \in \M_n$ is unitary, then $UA = [ U\vec{a}_1~U\vec{a}_2~\ldots~U\vec{a}_n]$.  Thus, \eqref{E:UIN-def-norm-F} ensures that
\begin{equation*}
\|UA\|_{\F}^2 
= \sum_{j=1}^{n} \|U \vec{a}_j\|^2 
= \sum_{j=1}^{n} \|\vec{a}_j\|^2 
= \|A\|_{\F}^2.
\end{equation*}
Since $V$ is unitary, the previous equation yields
\begin{equation*}
\| UAV \|_{\F}^2 
= \| (UAV)^{\T} \|_{\F}^2
= \| V^{\T} A^{\T} U^{\T} \|_{\F}^2
= \| A^{\T} U^{\T} \|_{\F}^2
= \| UA \|_{\F}^2
= \| A \|_{\F}^2.
\end{equation*}

According to \eqref{E:UIN-A=udiagV} and \eqref{E:UIN-def-norm-F}, we also see that
\begin{align}
\|A\|_{\F} 
&= \| \diag \big(s(A)\big)\| \notag\\
&= \bigg( \sum_{k=1}^{n} s_k^2(A) \bigg)^{\frac{1}{2}}. \label{E:UIN-spectral-norm-s2}
\end{align}

\section{Absolute and monotone norms}\label{S:UIN-absolute-norms}
For each $\vec{x} = [x_i] \in \C^n$, define
\begin{equation*}
|\vec{x}| :=
\begin{bmatrix}
|x_1| \\
|x_2| \\
\vdots\\
|x_n|
\end{bmatrix}
\qquad\text{and}\qquad
\vec{x}_{\sigma} :=
\begin{bmatrix}
x_{\sigma(1)} \\
x_{\sigma(2)} \\
\vdots\\
x_{\sigma(n)}
\end{bmatrix},
\end{equation*}
in which $\sigma$ is a permutation of $\{1,2,\ldots,n\}$. The set of all such permutations forms a group, the \emph{symmetric group} $\mathfrak{S}_n$, which has order $n!$ \cite{MR1721031, MR3496353, MR2963408}.

Given $\vec{x} = [x_i]$ and $\vec{y} = [y_i]$ in $\R^n$,
we write $\vec{x} \leq_{\e} \vec{y}$ if $x_j \leq y_j$ for all $1 \leq j \leq n$. We use a similar notation for matrices: given real matrices $A=[a_{ij}]$ and $B=[b_{ij}] \in \M_{m \times n}$, we write $A \leq_{\e} B$ if $a_{ij} \leq b_{ij}$ for all $1 \leq i \leq m$ and $1 \leq j \leq n$. The subscript ${}_\e$ indicates an \emph{entrywise} comparison to distinguish it from other orderings that might exist between matrices (for example, the L\"owner ordering in Hermitian matrices).

A norm $\|\cdot\|$ on $\C^n$ is \emph{monotone} if
\begin{equation*}
|\vec{x}| \leq_{\e} |\vec{y}| \implies
\|\vec{x}\| \leq \|\vec{y}\|,
\end{equation*}
\emph{absolute} if
\begin{equation*}
\big\|\, |\vec{x}|\,  \big\| = \|\vec{x}\|,
\end{equation*}
and \emph{symmetric} if
\begin{equation*}
\|\vec{x}_{\sigma}\| = \|\vec{x}\|
\end{equation*}
for all $\vec{x}, \vec{y} \in \C^n$ 
and $\sigma \in \mathfrak{S}_n$.
A \emph{symmetric gauge function} is an absolute, monotone, symmetric norm on $\R^n$. As a first step in our analysis, we show that the properties of being absolute and monotone are equivalent.

\begin{lemma} \label{L:UIN-monotone=absolute}
A norm is absolute if and only if it is monotone.
\end{lemma}

\begin{proof}
$(\Leftarrow)$ Let $\|\cdot\|$ be a norm on $\C^n$.
Let $\vec{x} \in \C^n$ and define $\vec{y}=|\vec{x}|$. Then $|\vec{x}| \leq_{\e} |\vec{y}|$ and $|\vec{y}| \leq_{\e} |\vec{x}|$. Therefore, if the norm is monotone, then $\|\vec{x}\|=\|\vec{y}\|$ and hence it is absolute.

\medskip\noindent$(\Rightarrow)$ Suppose that $\|\cdot\|$ is absolute.  By absoluteness, it suffices to consider $\vec{x} = [x_i]$ and $\vec{y} = [y_i] \in \R^n_{+}$
with $\vec{x} \leq_{\e} \vec{y}$. 
Let
\begin{equation*}
A = 
\begin{bmatrix}
x_1 & y_1 & y_1 & \cdots & y_1 & y_1 & y_1 \\
x_2 & x_2 & y_2 & \cdots & y_2 & y_2 & y_2 \\
x_3 & x_3 & x_3 & \cdots & y_3 & y_3 & y_3 \\
\vdots & \vdots & \vdots& \ddots & \vdots & \vdots & \vdots \\
x_{n-1} & x_{n-1} & x_{n-1} & \cdots & x_{n-1} & y_{n-1} & y_{n-1}\\
x_n & x_n & x_n & \cdots & x_n &  x_{n} & y_{n}\\
\end{bmatrix} \in \M_{n \times (n+1)}
\end{equation*}
and let $\vec{a}_1,\vec{a}_2,\ldots,\vec{a}_{n+1}$ denote the columns of $A$, so that $\vec{a}_1 = \vec{x}$ and $\vec{a}_{n+1} = \vec{y}$.  Then
$\vec{a}_j \leq_{\e} \vec{a}_{j+1}$ for all $1 \leq j \leq n$. From one column to the next, only one entry increases. Thus, it suffices to show that $\|\vec{a}_j\| \leq \|\vec{a}_{j+1}\|$ for $1 \leq j \leq n$. 

If $x_1=y_1$, the result is trivial. Suppose that $x_1<y_1$. We write $x_1$ as a convex combination of $y_1$ and $-y_1$:
\begin{equation*}
x_1 = t y_1 + (1-t) (-y_1),
\end{equation*}
in which $t = \frac{x_1+y_1}{2y_1}$ satisfies $0 < t < 1$. This allows us to write
\begin{equation*}
\begin{bmatrix}
x_1 \\
x_2 \\
\vdots\\
x_n
\end{bmatrix}
=
t\begin{bmatrix}
y_1 \\
x_2 \\
\vdots\\
x_n
\end{bmatrix}
+
(1-t)
\begin{bmatrix}
-y_1 \\
\phantom{-} x_2 \\
\phantom{-}\vdots\\
\phantom{-}x_n
\end{bmatrix}.
\end{equation*}
Therefore,
\begin{equation*}
\left\|\begin{bmatrix}
x_1 \\
x_2 \\
\vdots\\
x_n
\end{bmatrix}\right\|
\leq
t \left\|\begin{bmatrix}
y_1 \\
x_2 \\
\vdots\\
x_n
\end{bmatrix}\right\|
+
(1-t)
\left\|\begin{bmatrix}
-y_1 \\
\phantom{-}x_2 \\
\phantom{-}\vdots\\
\phantom{-}x_n
\end{bmatrix}\right\|.
\end{equation*}
Since the norm is absolute, the last two norms are equal
and hence $\|\vec{a}_1\| \leq \|\vec{a}_2\|$. Proceeding in a similar manner $n-1$ times yields $\|\vec{a}_j\| \leq \|\vec{a}_{j+1}\|$ as required.  Therefore, $\|\vec{x}\| = \| \vec{a}_1\| \leq \| \vec{a}_{n+1} \| = \| \vec{y}\|$.
\end{proof}

\section{Majorization and the Hardy--Littlewood--P\'{o}lya theorem} \label{S:UIN-majoraization}
The main purpose of this section and the next two is to establish Rad\'o’s result (Theorem \ref{T:UIN-Rado}), which is a major ingredient in the proof of the monotonicity theorem (Theorem \ref{T:UIN-monotonicity}). However, the path is long and requires several intermediate results, each of which is interesting in its own right.
For each $\vec{x} = [x_i] \in \R^n$, define
\begin{equation}\label{E:UIN-decreasing-order-x}
\vec{x}\!\!\downarrow \,\,:=
\begin{bmatrix}
x_{[1]} \\
x_{[2]} \\
\vdots\\
x_{[n]}
\end{bmatrix} \in \R^n,
\end{equation}
in which 
\begin{equation*}
x_{[1]} \geq x_{[2]} \geq \cdots \geq x_{[n]}.    
\end{equation*}
We say that $\vec{x} \in \R^n$ is \emph{weakly majorized} by $\vec{y} \in \R^n$, denoted $\vec{x} \prec_{\w} \vec{y}$, if
\begin{equation*}
\sum_{i=1}^{k} x_{[i]} \leq \sum_{i=1}^{k} y_{[i]}
\end{equation*}
for each $1 \leq k \leq n$.
If, in addition, equality holds for $k=n$, that is, if
\begin{equation*}
\sum_{i=1}^{n} x_{i} = \sum_{i=1}^{n} y_{i},
\end{equation*}
then $\vec{x}$ is \emph{majorized} by $\vec{y}$, denoted $\vec{x}\prec \vec{y}$. For example, if $\alpha_i \geq 0$ and $\sum_{i=1}^{n} \alpha_{i} =1$, then 
\begin{equation*}
\left( \frac{1}{n}, \, \frac{1}{n}, \, \dots, \, \frac{1}{n} \right)
\prec
(\alpha_1,\alpha_2,\ldots,\alpha_n)
\prec(1,0,\ldots,0).
\end{equation*}

A matrix $A=[a_{ij}] \in \M_n$ is \emph{doubly stochastic} if each $a_{ij} \geq 0$ and if
\begin{equation*}
\sum_{k=1}^{n} a_{ik} = 1
\quad\text{and}\quad
\sum_{k=1}^{n} a_{kj} = 1
\end{equation*}
for all $1 \leq i,j \leq n$; that is,
the entries are nonnegative and each row and column sums to $1$.  We denote the set of all $n \times n$ doubly stochastic matrices by $\DS_n$. Then $A \in \DS_n$ if and only if $A\vec{1} = \vec{1}$ and $A^{\T}\vec{1} = \vec{1}$.  Therefore, $\DS_n$ is closed under matrix multiplication. 

Given $\sigma \in \mathfrak{S}_n$, define $P_{\sigma}=[p_{ij}] \in \M_n$ by
\begin{equation*}
p_{ij} =
\begin{cases}
1 & \text{if $\sigma(i)=j$},\\
0 & \text{if $\sigma(i) \neq j$}.
\end{cases} 
\end{equation*}
Each such \emph{permutation matrix} is doubly stochastic.  In fact, each row and each column contains exactly one nonzero element, namely $1$. The operation $A \mapsto P_{\sigma}A$ interchanges the rows of $A$ and the operation $A \mapsto AP_{\sigma}$ interchanges the columns of $A$. Given $\vec{x} \in \R^n$, there is a $\sigma \in S_n$ such that 
$\vec{x}\!\!\downarrow\,\,  = P_{\sigma}\vec{x}$; see \cite{MR1449393, MR932967, MR389944}.

The following result of Hardy--Littlewood--P\'{o}lya establishes the connection between majorization and doubly stochastic matrices. It was originally published in their 1929 paper \cite{HLP-1929} and later refined in their 1934 book \cite{MR944909}.

\begin{theorem}[Hardy--Littlewood--P\'{o}lya] \label{T:UIN-Hardy-Littlewood-Polya}
Let $\vec{x},\vec{y} \in \R^n$. Then $\vec{x} \prec \vec{y}$ if and only if there is an $A \in \DS_n$ such that $\vec{x} = A\vec{y}$.
\end{theorem}

\begin{proof}
$(\Rightarrow)$
Suppose that $\vec{x} \prec \vec{y}$. We use a procedure similar to that in the proof of Lemma \ref{L:UIN-monotone=absolute}. First, we multiply the vector $\vec{y}$ by a suitable $A_1 \in \DS_n$ such that
\begin{equation}\label{E:UIN-matrix A1}
A_1
\begin{bmatrix}
y_1 \\
y_2 \\
\vdots\\
y_n
\end{bmatrix}
=
\begin{bmatrix}
x_1 \\
y_2' \\
\vdots\\
y_n'
\end{bmatrix}.
\end{equation}
Moreover, $A_1$ must be chosen such that $[y_2'~y_3'~\ldots~y_n']^{\T}$ satisfies
\begin{equation}\label{E:UIN-matrix A12}
\begin{bmatrix}
x_2 \\
\vdots\\
x_n
\end{bmatrix}
\prec
\begin{bmatrix}
y_2' \\
\vdots\\
y_n'
\end{bmatrix}.
\end{equation}
If this can be established, then the rest follows by induction.  


Let us explain how $A_1$ is chosen. Without loss of generality, suppose that $x_1 \geq x_2 \geq \cdots \geq x_n$ and $y_1 \geq y_2 \geq \cdots \geq y_n$. Indeed, otherwise there are permutations $\sigma,\tau \in \mathfrak{S}_n$ such that $\vec{x}\!\!\downarrow\,\, = P_{\sigma}\vec{x}$ and $\vec{y}\!\!\downarrow\,\, = P_{\tau}\vec{y}$.  Then we can establish the result for $\vec{x}\!\!\downarrow$ and $\vec{y}\!\!\downarrow$ instead. Therefore, there will be an $A\in \DS_n$ such that $\vec{x}\!\!\downarrow\,\, = A \vec{y}\!\!\downarrow$ and hence $\vec{x} = P_{\sigma}^{\T}AP_{\tau}\vec{y}$. Note that $P_{\sigma}^{\T}AP_{\tau}$ is doubly stochastic.

Since $\vec{x} \prec \vec{y}$, we have $y_n \leq x_1 \leq y_1$ because the entries of $\vec{x}$ and $\vec{y}$ are decreasing; the lower inequality follows from the equality of the total sums and the ordering of the terms. Thus, there is a $2 \leq k \leq n$ such that $y_k \leq x_1 \leq y_{k-1}$, so
\begin{equation}\label{E:UIN-matrix 3}
x_1 = ty_1+(1-t)y_k,
\end{equation}
in which $0 \leq t \leq 1$. Let $A_1=[a_{ij}] \in \M_n$ with
\begin{equation*}
    a_{ij} =
    \begin{cases}
        t & \text{if $(i,j) \in  \{(1,1), (k,k)\}$},\\
        1-t & \text{if $(i,j) \in \{(1,k),(k,1)\}$},\\
        1 & \text{if $i =j \notin \{1,k\}$},\\
        0 & \text{otherwise}.
    \end{cases}
\end{equation*}
In terms of permutations, $A_1=tI+(1-t)P_{(1k)}$, in which $(1k)\in \mathfrak{S}_n$ denotes a transposition. One can now verify that \eqref{E:UIN-matrix A1} holds. More explicitly, \eqref{E:UIN-matrix 3} ensures that
\begin{equation}\label{E:UIN-matrix 4}
A_1
\begin{bmatrix}
y_1 \\
y_2 \\
\vdots\\
y_{k-1} \\
y_{k} \\
y_{k+1} \\
\vdots\\
y_n
\end{bmatrix}
=
\begin{bmatrix}
\color{blue}x_1 \\
y_2 \\
\vdots\\
y_{k-1} \\
\color{blue} (1-t)y_1+ty_{k} \\
y_{k+1} \\
\vdots\\
y_n
\end{bmatrix}.
\end{equation}
This observation facilitates the verification of \eqref{E:UIN-matrix A12}. In other words, we must show that
\begin{equation*}
\begin{bmatrix}
x_2 \\
\vdots\\
x_{k-1} \\
x_{k} \\
x_{k+1} \\
\vdots\\
x_n
\end{bmatrix}
\prec
\begin{bmatrix}
y_2 \\
\vdots\\
y_{k-1} \\
(1-t)y_1+ty_{k} \\
y_{k+1} \\
\vdots\\
y_n
\end{bmatrix}.
\end{equation*}

Since $y_1 \geq \cdots \geq y_{k-1} \geq x_1 \geq x_2 \geq \cdots \geq x_n$, we have
\begin{equation*}
\sum_{j=2}^{m} x_j \leq (m-1)x_1 \leq \sum_{j=2}^{m} y_j
\end{equation*}
for $m=2,3,\ldots,k-1$. Then for $m=k,k+1,\ldots,n$, we have
\begin{align*}
\sum_{j=2}^{m} x_j &= -x_1 + \sum_{j=1}^{m} x_j\\
&\leq -x_1 + \sum_{j=1}^{m} y_j && (\text{since $x\prec y$})\\
&= -ty_1-(1-t)y_k + \sum_{j=1}^{m} y_j &&  (\text{by \eqref{E:UIN-matrix 3}})\\
&= \sum_{j=2}^{k-1} y_j +\big((1-t)y_1+ty_k\big) + \sum_{j=k+1}^{m} y_j.
\end{align*}
If $m=n$, the inequality above is an equality.  Thus, \eqref{E:UIN-matrix A12} holds.

We can sum up using an induction argument. By induction hypothesis (which trivially holds for the case $n=1$), there is a $B \in \DS_{n-1}$ such that
\begin{equation*}
B
\begin{bmatrix}
y_2 \\
\vdots\\
y_{k-1} \\
(1-t)y_1+ty_{k} \\
y_{k+1} \\
\vdots\\
y_n
\end{bmatrix}
=
\begin{bmatrix}
x_2 \\
\vdots\\
x_{k-1} \\
x_{k} \\
x_{k+1} \\
\vdots\\
x_n
\end{bmatrix}.
\end{equation*}
Let $C = \diag (1,B) \in \M_n$ and observe that
\begin{equation*}
C
\begin{bmatrix}
x_1\\
y_2 \\
\vdots\\
y_{k-1} \\
(1-t)y_1+ty_{k} \\
y_{k+1} \\
\vdots\\
y_n
\end{bmatrix}
=
\begin{bmatrix}
x_1\\
x_2 \\
\vdots\\
x_{k-1} \\
x_{k} \\
x_{k+1} \\
\vdots\\
x_n
\end{bmatrix}.
\end{equation*}
According to \eqref{E:UIN-matrix 4}, we have $\vec{x} = A\vec{y}$ with $A=CA_1$.

\medskip\noindent$(\Leftarrow)$ Suppose that there is an $A=[a_{ij}] \in \DS_n$ such that $\vec{x}=A\vec{y}$. Without loss of generality, we assume that $\vec{x}=\vec{x}\!\!\downarrow$ and $\vec{y}=\vec{y}\!\!\downarrow$ by the same permutation argument used above. For $m=1,2,\ldots,n$, we have
\begin{equation*}
\sum_{i=1}^{m} x_i
= \sum_{i=1}^{m} \bigg( \sum_{j=1}^{n} a_{ij} y_j \bigg)
= \sum_{j=1}^{n} \bigg( \sum_{i=1}^{m} a_{ij} \bigg) y_j.
\end{equation*}
Let
\begin{equation*}
t_j^{(m)} = \sum_{i=1}^{m} a_{ij}
\end{equation*}
for $1 \leq j \leq n$. Then note that $0 \leq t_j^{(m)} \leq 1$ and
\begin{equation*}
\sum_{j=1}^{n} t_j^{(m)} 
= \sum_{j=1}^{n} \bigg( \sum_{i=1}^{m} a_{ij} \bigg) 
= \sum_{i=1}^{m} \bigg(  \sum_{j=1}^{n} a_{ij} \bigg) 
=\sum_{i=1}^{m}  1 = m
\end{equation*}
since each $a_{ij} \geq 0$ and the rows of $A$ sum to $1$. Then
\begin{align*}
\sum_{i=1}^{m} x_i &= \sum_{j=1}^{n} t_j^{(m)} y_j\\
&= \sum_{j=1}^{n} t_j^{(m)} (y_j-y_m) + my_m\\
&= \sum_{j=1}^{m} t_j^{(m)} (y_j-y_m) + \sum_{j=m+1}^{n} t_j^{(m)} (y_j-y_m) + \sum_{j=1}^{m} y_m,
\end{align*}
which gives
\begin{equation*}
\sum_{i=1}^{m} x_i-\sum_{i=1}^{m} y_i = \sum_{j=1}^{m} (t_j^{(m)}-1) (y_j-y_m) + \sum_{j=m+1}^{n} t_j^{(m)} (y_j-y_m) \leq 0.
\end{equation*}
If $m=n$, we have $t_j^{(n)}=1$ for each $j$ since
the columns of $A$ sum to $1$.  Thus,
\begin{equation*}
\sum_{i=1}^{n} x_i = \sum_{i=1}^{n} y_i
\end{equation*}
and hence $\vec{x} \prec \vec{y}$.
\end{proof}

\section{K\"{o}nig's term-rank theorem} \label{S:UIN-konig}
Let $A \in \M_{m \times n}$. A row or a column of $A$ is a \emph{line}. A set of lines \emph{covers} $A$ if the set contains all nonzero elements of $A$. In general, any such set consists of some rows and some columns. However, one could choose only rows or only columns. For example, a line set consisting of all $n$ columns of $A$ covers $A$. However, there are other options. The minimum number of lines that covers $A$ is the \emph{line rank} of $A$, denoted $\ell(A)$. A minimal covering is \emph{proper} if it does not consist of all columns and it does not consist of all rows. For example, for the $2\times 2$ identity matrix 
\begin{equation*}
I_2 = [\vec{c}_1\,\,\vec{c}_2] 
= \Big[ \begin{smallmatrix} \vec{r}_1^{\T} \\ \vec{r}_2^{\T} \end{smallmatrix} \Big],    
\end{equation*}
the coverings $\{\vec{r}_1^{\T},\vec{r}_2^{\T}\}$, $\{\vec{c}_1,\vec{c}_2\}$, $\{\vec{r}^{\T}_1,\vec{c}_2\}$, are minimal; the last one is proper.

A related concept is the \emph{term rank} of $A$, denoted by $\tau(A)$, which is the maximum number of nonzero entries in $A$ with no two of them on the same line. The definitions imply that $\tau(A) \leq \ell(A)$. A result of K\"{o}nig ensures that equality holds, a profound observation with interesting consequences. He obtained precursor results in 1916 \cite{Konig-1}. The following result appeared in 1931 \cite{Konig-2}.

\begin{theorem}[K\"{o}nig] \label{T:UIN-konig}
For every $A \in \M_{m \times n}$, 
\begin{equation*}
\tau(A) = \ell(A).
\end{equation*}
\end{theorem}

\begin{proof}
It suffices to prove that $\tau(A) \geq \ell(A)$. We induct on $m+n$. The result holds for $m=1$ (row vectors), $n=1$ (column vectors), and zero matrices. So assume that $m \geq 2$ and $n \geq 2$ and that the result holds for matrices whose total number of rows and columns is at most $m+n-1$. There are two cases.

\medskip
\noindent\textsc{Case 1}: Suppose that $A=[a_{ij}]$ does not have a proper minimal covering.  Then $\ell(A) = \min\{m,n\}$. Pick any entry $a_{ij} \neq 0$ and let $A' \in \M_{(m-1) \times (n-1)}$ be 
obtained from $A$ by deleting its $i$th row and $j$th column. Hence,
\begin{equation*}
\ell(A') \leq \min\{m-1,n-1\} = \ell(A)-1.
\end{equation*}
We must have $\ell(A') = \ell(A)-1$ since if $\ell(A') \leq \ell(A)-2$, then a minimal covering of $A'$ augmented by the two deleted lines of $A$ would give a proper covering of $A$, contrary to our assumption (this covering of $A$ would have at most $m-1$ rows, and at most $n-1$ columns). By the induction hypothesis, $\tau(A')=\ell(A')=\ell(A)-1$. Picking $\tau(A')$ elements of $A'$ so that no two of them are on the same line of $A'$ and adding $a_{ij}$ to the collection, we obtain $1+\tau(A')$ elements of $A$, no two on the same line of $A$. Hence, $\tau(A) \geq 1+\tau(A')=\ell(A)$.

\medskip\noindent\textsc{Case 2}: Assume that $A$ has a proper minimal covering that consists of $p$ rows and $q$ columns, in which $\ell(A)=p+q$, $p<m$, $q<n$, and $p+q \leq \min\{m,n\}$. If $p=0$ or $q=0$, then the result follows from the induction hypothesis. For example, if $p=0$, the $q$ selected columns contain all nonzero entries; delete the remaining zero columns and apply induction to the resulting $m\times q$ matrix, whose line rank is $q$. The $q=0$ case is analogous.

If $p \geq 1$ and $q \geq 1$, we multiply $A$ by permutation matrices on the left and right so that the $p$ rows and the $q$ columns of the proper minimal covering appear as the first $p$ rows and the first $q$ columns. This action does not change the term rank of $A$. Thus, $A$ takes the form
\begin{equation*}
\begin{bmatrix}
* & B\\
C & 0
\end{bmatrix},
\end{equation*}
in which $B\in \M_{p \times (n-q)}$ and $C \in \M_{(m-p) \times q}$.  The lower-right corner is an $(m-p)\times(n-q)$ zero matrix. Recall that $p \leq n-q$ and $q \leq m-p$. We must have $\ell(B) = p$ and $\ell(C)=q$, since if, for example, $\ell(B)<p$, then we can cover $A$ by the first $q$ columns and another $\ell(B)$ lines to cover $B$, a total of less than $p+q$ lines, a contradiction. Therefore, by the induction hypothesis $\tau(B)=p$ and $\tau(C)=q$, giving
\begin{equation*}
\tau(A) \geq \tau(B)+\tau(C)=p+q=\ell(A). \qedhere
\end{equation*}
\end{proof}

Let $A=[a_{ij}] \in \M_n$. For each $\sigma \in S_n$, the sequence
\begin{equation*}
(a_{1\sigma(1)},\, a_{2\sigma(2)},\, \dots, \, a_{n\sigma(n)})
\end{equation*}
is a \emph{transversal} of $A$. If $\sigma=\mathrm{id}$, we obtain the diagonal of $A$.  Thus, we think of transversals as  generalized diagonals of $A$.

\begin{corollary} \label{C:UIN-konig}
Let $A=[a_{ij}] \in \M_n$. Suppose that for every $r \times s$ zero submatrix of $A$ we have $r+s \leq n$. Then $A$ has a transversal with nonzero entries.
\end{corollary}

\begin{proof}
Suppose that every transversal of $A$ has at least one zero entry. Hence, $\tau(A)\leq n-1$. Theorem \ref{T:UIN-konig} ensures that $\ell(A) \leq n-1$. Deleting the $\ell(A)$ lines of a minimal covering, we obtain an $r \times s$ zero submatrix with
\begin{equation*}
r+s = 2n-\ell(A) \geq n+1.
\end{equation*}
This is a contradiction. Therefore, there must be at least one transversal with nonzero entries.
\end{proof}

\section{Convexity: the Birkhoff and Rad\'o theorems} \label{S:UIN-birkhoff}
Let $S \subseteq \R^n$ be convex. A point $\vec{x} \in S$ is an \emph{extreme point} of $S$
if it is not in the interior of any nondegenerate line segment in $S$.  That is, 
if $\vec{y},\vec{z} \in S$ and $\vec{x}=t \vec{y}+(1-t)\vec{z}$ with $0<t<1$, then $\vec{x}=\vec{y}=\vec{z}$. The set of extreme points of $S$ is denoted by $\ex(S)$. Note that $\ex(S)$ might be empty; for example, consider the open unit ball. The Krein--Milman theorem ensures that if $S$ is nonempty, compact, and convex, then $\ex(S)$ is not empty and rich enough to reconstruct $S$ in the following sense. The set of all convex combinations of elements of $S$ is its \emph{convex hull}, and is denoted $\conv(S)$. If $K \subset \R^n$ is compact and convex, the simplest version of the Krein--Milman theorem \cite{MR4990} says that $\ex(K) \neq \varnothing$ and, moreover, $\conv(\ex(K)) = K$. The Krein--Milman theorem has several variants and generalizations \cite{MR1068530, MR1157815}.

Recall that $\DS_n$ denotes the set of all $n \times n$ doubly stochastic matrices. Let $\Perm_n$ be the set of all $n \times n$ permutation matrices.  Then $\Perm_n \subset \DS_n$ and $\DS_n$ is a compact convex subset of $\R^{n^2}$. The Krein--Milman theorem ensures that $\DS_n$ is the convex hull of its extreme points. Birkhoff's result further clarifies this fact \cite{MR20547}; see also \cite{MR153038}.

\begin{theorem}[Birkhoff \cite{MR20547}] \label{T:UIN-birkhoff}
We have
\begin{equation*}
\ex(\DS_n) = \Perm_n,
\end{equation*}
and hence every doubly stochastic matrix is a convex combination of permutation matrices.
\end{theorem}

Before giving the proof of Birkhoff's theorem, we consider an example.

\begin{example}
Let
\begin{equation*}
A =
\begin{bmatrix}
0 & 1/2 & 1/2\\
1/3 & 1/3 & 1/3\\
2/3 & 1/6 & 1/6
\end{bmatrix} \in \DS_3.
\end{equation*}
Pick a transversal of $A$ with nonzero entries (why does such a transversal exist in the general setting?). For $A$, we have several choices. We take the one indicated by blue below
\begin{equation*}
\begin{bmatrix}
\0 & {\color{blue}1/2} & 1/2 \\
{\color{blue}1/3} & 1/3 & 1/3 \\
2/3 & 1/6 & {\color{blue}1/6}
\end{bmatrix},
\end{equation*}
which corresponds to the permutation matrix
\begin{equation*}
P_1=
\begin{bmatrix}
\0 & 1 & \0\\
1 & \0 & \0\\
\0 & \0 & 1
\end{bmatrix}.
\end{equation*}
The minimum entry in this transversal is $1/6$, so we define
\begin{equation*}
B = \frac{A-\frac{1}{6}P_1}{1-\frac{1}{6}} =
\begin{bmatrix}
\0 & 2/5 & 3/5\\
1/5 & 2/5 & 2/5\\
4/5 & 1/5 & \0
\end{bmatrix}.
\end{equation*}
We see that $B$ is a doubly stochastic matrix, and in comparison to $A$, it has one more zero entry. We rewrite the previous identity as
\begin{equation}\label{E:UIN-convex-sumA-1}
A = \tfrac{1}{6}P_1 + \tfrac{5}{6}B.
\end{equation}
We continue this procedure with $B$, and at each step increase the number of zero entries. For example, in $B$, choose the transversal
\begin{equation*}
\begin{bmatrix}
\0 & {\color{blue}2/5} & 3/5\\
1/5 & 2/5 & {\color{blue}2/5}\\
{\color{blue}4/5} & 1/5 & \0
\end{bmatrix}
\end{equation*}
corresponding to the permutation
\begin{equation*}
P_2 =
\begin{bmatrix}
\0 & 1 & \0\\
\0 & \0 & 1\\
1 & \0 & \0
\end{bmatrix},
\end{equation*}
and define
\begin{equation*}
C = \frac{B-\frac{2}{5}P_2}{1-\frac{2}{5}} =
\begin{bmatrix}
\0 & \0 & 1\\
1/3 & 2/3 & \0\\
2/3 & 1/3 & \0
\end{bmatrix}.
\end{equation*}
Then $B= \frac{2}{5}P_2+\frac{3}{5}C$ and hence \eqref{E:UIN-convex-sumA-1} gives
\begin{equation}\label{E:UIN-convex-sumA-2}
A = \tfrac{1}{6}P_1 + \tfrac{1}{3}P_2+\tfrac{1}{2}C.
\end{equation}
In this final step, consider the transversal 
\begin{equation*}
\begin{bmatrix}
\0 & \0 & {\color{blue}1}\\
1/3 & {\color{blue}2/3} & \0\\
{\color{blue}2/3} & 1/3 & \0
\end{bmatrix},
\end{equation*}
which corresponds to
\begin{equation*}
P_3 =
\begin{bmatrix}
\0 & \0 & 1\\
\0 & 1 & \0\\
1 & 0 & \0
\end{bmatrix},
\end{equation*}
and define
\begin{equation*}
D = \frac{C-\frac{2}{3}P_3}{1-\frac{2}{3}}
=
\begin{bmatrix}
\0 & \0 & 1\\
1 & \0 & \0\\
\0 & 1 & \0
\end{bmatrix} = P_4.
\end{equation*}
Thus, $C=\frac{2}{3}P_3+\frac{1}{3}P_4$, from which \eqref{E:UIN-convex-sumA-2} yields
\begin{equation*}
A = \tfrac{1}{6}P_1 + \tfrac{1}{3}P_2+\tfrac{1}{3}P_3+\tfrac{1}{6}P_4.
\end{equation*}
\end{example}

We are now ready for the proof of Birkhoff's theorem. The crucial step is the use of Corollary \ref{C:UIN-konig} to obtain a transversal with all entries nonzero.    

\begin{proof}[Proof of Theorem \ref{T:UIN-birkhoff}]
We first show that each permutation matrix $P=[p_{ij}] \in \Perm_n$ is an extreme point of $\DS_n$. 
Suppose that $P=tA+(1-t)B$, in which $A=[a_{ij}]$ and $B=[b_{ij}] \in \DS_n$ and $0<t<1$. Since the entries of $P$, $A$, and $B$ are in $[0,1]$, and since $0<t<1$, it follows that $p_{ij}=0$ implies that $a_{ij}=b_{ij}=0$. Similarly, wherever $p_{ij}=1$, we must have $a_{ij}=b_{ij}=1$. Thus, $P=A=B$ and hence $P$ is an extreme point of $\DS_n$; that is,
\begin{equation*}
\Perm_n \subset \ex(\DS_n).
\end{equation*}

Suppose that $A \in \DS_n$. If all entries of $A$ are positive, then every transversal of $A$ consists entirely of nonzero entries. Assume that $A$ has at least one zero entry. By multiplying $A$ on both sides by appropriate permutation matrices, we can move any $r \times s$ zero submatrix of $A$ to the upper-left corner. The new matrix, still doubly stochastic, has the form
\begin{equation*}
A' = 
\begin{bmatrix}
0 & X\\
Y & Z
\end{bmatrix},
\end{equation*}
in which the zero matrix above is $r \times s$. The sum of every entry of $A'$ is $n$. Since each row of $X$ adds up to $1$, the sum of all elements of $X$ is $r$. Similarly, since each column of $Y$ also adds up to 1, the sum of all elements of $Y$ is $s$. Therefore,
\begin{equation*}
r+s \leq n.
\end{equation*}
Corollary \ref{C:UIN-konig} ensures that $A'$, and hence $A$, has a transversal with all entries nonzero. 

Fix any transversal of $A$, with the corresponding permutation $P$, so that all its entries are positive. Let $t$ be the minimum entry in the transversal. If $t=1$, then $A=P$ and we are done. If $0<t<1$, let
\begin{equation*}
B = \frac{A-tP}{1-t}.
\end{equation*}
Since $t$ is the smallest element of $A$ along the chosen transversal, the entries of $B$ are nonnegative. Moreover, the division by $1-t$ ensures that every row and column of $B$ sums to $1$; that is, $B$ is doubly stochastic. At entries $(i,j)$ where the transversal takes the minimum value $t$, we have $b_{ij}=0$. Therefore, we have the convex combination
\begin{equation*}
A = tP+(1-t)B,
\end{equation*}
in which $P \in \Perm_n$ and $B$ has at least one more zero entry than $A$. Repeating this procedure a finite number of times, we end up with a representation of $A$ as a convex combination of permutations.

The reasoning above shows that the permutations are the only extreme points of $\DS_n$. Indeed, suppose that $A \in \DS_n$ is not a permutation. By the previous paragraph, there are distinct permutation matrices $P_1,P_2,\ldots,P_N$, in which $N \geq 2$, and $0<t_j<1$ with $\sum_{j=1}^{N}t_j=1$ such that
\begin{equation*}
A = \sum_{j=1}^{N} t_j P_j.
\end{equation*}
If $N=2$, we are done since $A$ is in the interior of the line segment determined by $P_1$ and $P_2$. 
If $N \geq 3$, then we write
\begin{equation*}
A = t_1P_1 + (1-t_1) \bigg(\sum_{j=2}^{N} \frac{t_j}{t} P_j\bigg),
\end{equation*}
in which $t=1-t_1=\sum_{j=2}^{N}t_j$. This is a convex combination of $P_1$ and the doubly stochastic matrix
\begin{equation*}
B=\sum_{j=2}^{N} \frac{t_j}{t} P_j \in \DS_n.
\end{equation*}
Moreover, $B \neq P_1$ since otherwise $A = P_1$, which contradicts the assumption that $A$ is not a permutation matrix.  Since $0 < t_1 < 1$, it follows that $A$ lies in the interior of the nondegenerate line segment joining $P_1$ and $B$, so $A$ is not extreme. 
In other words, $\ex(\DS_n) \subset \Perm_n$.     
\end{proof}

The following result is a direct consequence of Theorems \ref{T:UIN-Hardy-Littlewood-Polya} and \ref{T:UIN-birkhoff}.

\begin{theorem}[Rad\'o \cite{MR45168}] \label{T:UIN-Rado}
Let $\vec{x},\vec{y} \in \R^n$. Then $\vec{x} \prec \vec{y}$ if and only if $\vec{x}$ is a convex combination of permutations of $\vec{y}$.
\end{theorem}

\begin{proof}
By the Hardy--Littlewood--P\'{o}lya theorem (Theorem \ref{T:UIN-Hardy-Littlewood-Polya}), it follows that $\vec{x} \prec \vec{y}$ if and only if there is an $A \in \DS_n$ such that $\vec{x} = A\vec{y}$. By Birkhoff's theorem  (Theorem \ref{T:UIN-birkhoff}), $A \in \DS_n$ if and only if there are $P_j \in \Perm_n$ and $t_1,t_2,\ldots,t_N \geq 0$
such that $t_1+t_2+\cdots+t_N= 1$ and
\begin{equation*}
A = \sum_{j=1}^{N} t_j P_j.
\end{equation*}
Hence, $\vec{x} \prec \vec{y}$ if and only if
\begin{equation*}
\vec{x} = A\vec{y} = \sum_{j=1}^{N} t_j P_j\vec{y}.
\end{equation*}
Note that each $P_j \vec{y}$ is a vector obtained by permuting the entries of $\vec{y}$.
\end{proof}

\section{Weak majorization} \label{S:UIN-weak-major}

The Hardy--Littlewood--P\'{o}lya theorem (Theorem \ref{T:UIN-Hardy-Littlewood-Polya}) characterizes the majorization $\vec{x} \prec \vec{y}$ via the existence of a doubly stochastic matrix $A \in \M_n$ such that $\vec{x} = A\vec{y}$. Then Birkhoff's theorem (Theorem \ref{T:UIN-birkhoff}) shows that the extreme points of this family are precisely the permutation matrices. Therefore, it is natural to think of a weak version of doubly stochastic matrices to characterize weak majorization, and then a weak version of permutations to obtain the extreme points of this class.

A nonnegative matrix $A \in \M_n$ is \emph{doubly substochastic} if the sum of the entries in each row and each column is at most $1$. The set of all doubly substochastic matrices in $\M_n$ is denoted by $\DSS_n$. One can enlarge a doubly substochastic matrix to a doubly stochastic matrix and exploit the properties of the doubly stochastic matrices to derive results for doubly substochastic matrices. Let $A \in \DSS_n$ and define $R=\diag (r_1,r_2,\ldots,r_n)$ and $C=\diag (c_1,c_2,\ldots,c_n)$, in which $r_i$ is the sum of the entries in the $i$th row of $A$ and $c_j$ is the sum of the entries in the $j$th column of $A$. Then 
\begin{equation}\label{E:A-to-boubly}
\begin{bmatrix}
A & I_n-R \\
I_n-C & A^{\T}
\end{bmatrix} \in \DS_{2n}.
\end{equation}

We say that $A \in \M_n$ is a \emph{sub-permutation}, or a \emph{weak permutation}, if in every row and in every column of $A$ there is at most one nonzero entry, which must be $1$, and all other entries are $0$. Thus, a sub-permutation matrix is obtained by picking a permutation matrix and then changing some of its entries from $1$ to $0$. The family of all sub-permutations in $\M_n$ is denoted by $\PermSub_n$. In light of \eqref{E:A-to-boubly} and Birkhoff's theorem (Theorem \ref{T:UIN-birkhoff}),
\begin{equation}\label{E:UIN-sub-birkhoff}
\ex(\DSS_n) = \PermSub_n
\end{equation}
since each sub-permutation matrix is extreme.
Indeed, if $P = tA + (1-t)B$, in which $A,B \in \DSS_n$ and $0<t<1$, then $p_{ij}=0$ ensures that $a_{ij}=b_{ij}=0$; if $p_{ij} = 1$, the row- and column-sum bounds ensure that $a_{ij},b_{ij}\leq 1$,
and hence $a_{ij}=b_{ij}=1$ since a convex combination of them equals $1$.

This leads to a characterization of doubly substochastic matrices. Since we can change a sub-permutation matrix to a permutation matrix by changing some of its $0$ entries to $1$s, the representation of a doubly substochastic matrix as a convex combination of sub-permutations, implicitly given in \eqref{E:UIN-sub-birkhoff}, shows that
\begin{equation}\label{E:UIN-char-sub-double}
A \in \DSS_n
\iff
\text{there exists a $B \in \DS_n$ such that $0 \leq_{\e} A \leq_{\e} B$}.
\end{equation}

We are now ready to provide a characterization of weak majorization with the same flavor as the Hardy--Littlewood--P\'{o}lya theorem (Theorem \ref{T:UIN-Hardy-Littlewood-Polya}). We present it in two steps since the results are slightly different depending on whether the entries of certain vectors are all nonnegative or not.

\begin{theorem}\label{T:UIN-weak-major-char-1}
Let $\vec{x},\vec{y} \in \R^n_{+}$. Then $\vec{x} \prec_{\w} \vec{y}$ if and only if there is $A \in \DSS_n$ such that $\vec{x} = A\vec{y}$.
\end{theorem}

\begin{proof}
$(\Rightarrow)$
Suppose that $\vec{x} \prec_{\w} \vec{y}$. First, we consider two extreme cases. 
If $\vec{x}=\vec{0}$, then let $A=0$. If $\vec{x}\prec \vec{y}$, then the Hardy--Littlewood--P\'{o}lya theorem (Theorem \ref{T:UIN-Hardy-Littlewood-Polya}) provides  an $A \in \DS_n$ such that $\vec{x}=A\vec{y}$. Suppose that neither of these cases holds and let
\begin{equation*}
s = \sum_{k=1}^{n} y_k - \sum_{k=1}^{n} x_k.
\end{equation*}
Since $\vec{x} \nprec \vec{y}$, we have $s > 0$. Let
\begin{equation*}
x_{\min} = \min\{x_k: x_k>0\}
\qquad\text{and}\qquad
y_{\min} = \min\{y_k: y_k>0\};
\end{equation*}
these quantities are positive since $\vec{x} \neq \vec{0}$ and $\vec{x} \prec_{\w} \vec{y}$. Choose $m$ large enough so that
\begin{equation*}
\frac{s}{m} \leq x_{\min}
\qquad\text{and}\qquad
\frac{s}{m} \leq y_{\min},
\end{equation*}
and then define $a = s/m$,
\begin{equation*}
\widehat{\vec{x}} =
\begin{bmatrix}
x_1\\
\vdots\\
x_n\\
a\\
\vdots\\
a
\end{bmatrix},
\qquad\text{and}\qquad
\widehat{\vec{y}} =
\begin{bmatrix}
y_1\\
\vdots\\
y_n\\0\\
\vdots\\
0
\end{bmatrix} \in \R^{n+m}.
\end{equation*}
Since $a\leq x_{\min}$, the $a$'s occur after the positive entries of $\vec{x}$ in $\widehat{\vec{x}}$. Up to that point, the required inequalities are exactly $\vec{x}\prec_{\w} \vec{y}$. Every subsequent partial sum of $\widehat{\vec{x}}$ is at most the total sum, which is $\sum_{k=1}^n y_k =\sum_{k=1}^n \widehat{y}_k$, from
which it follows that $\widehat{\vec{x}} \prec \widehat{\vec{y}}$. 
The Hardy--Littlewood--P\'{o}lya theorem (Theorem \ref{T:UIN-Hardy-Littlewood-Polya}) provides an $\widehat{A} \in \DS_{n+m}$ such that $\widehat{\vec{x}} = \widehat{A} \widehat{\vec{y}}$. Decompose $\widehat{A}$ as
\begin{equation*}
\widehat{A} =
\begin{bmatrix}
A & *\\
* & *
\end{bmatrix},
\end{equation*}
in which $A$ is the upper-left $n \times n$ submatrix, which is in $\DSS_n$. The identity $\widehat{\vec{x}} = \widehat{A} \widehat{\vec{y}}$ yields $\vec{x}=A\vec{y}$.

\medskip\noindent$(\Leftarrow)$
Suppose that  $\vec{x} = A\vec{y}$, in which $A \in \DSS_n$. Then \eqref{E:UIN-char-sub-double} provides
a $B \in \DS_n$ such that $A \leq_{\e} B$. Hence, $\vec{x} = A\vec{y} \leq_{\e} B\vec{y}$ (positivity is used here) and, by the Hardy--Littlewood--P\'{o}lya theorem, $B\vec{y} \prec \vec{y}$. Therefore, $\vec{x} \prec_{\w} \vec{y}$.
\end{proof}

For $\vec{x},\vec{y} \in \R^n_{+}$, the proof of Theorem \ref{T:UIN-weak-major-char-1} shows that 
$\vec{x} \prec_{\w} \vec{y}$ if and only if there is a $\vec{z}=B\vec{y} \in \R^n_{+}$ such that $\vec{x} \leq_{\e} \vec{z}$ and $\vec{z} \prec \vec{y}$. If the entries of $x$ and $y$ are not necessarily nonnegative, we need to relax the conditions to handle the new setting.

\begin{theorem}\label{T:UIN-weak-major-char-2}
Let $\vec{x},\vec{y} \in \R^n$. Then $\vec{x} \prec_{\w} \vec{y}$ if and only if there is a $\vec{z} \in \R^n$ such that $\vec{x} \leq_{\e} \vec{z}$ and $\vec{z} \prec \vec{y}$.
\end{theorem}

\begin{proof}
$(\Rightarrow)$
Suppose that $\vec{x} \prec_{\w} \vec{y}$. Choose $t>0$ so large that all entries of
\begin{equation*}
\widetilde{\vec{x}} =
\begin{bmatrix}
x_1+t\\
x_2+t\\
\vdots\\
x_n+t
\end{bmatrix}
\qquad\text{and}\qquad
\widetilde{\vec{y}} =
\begin{bmatrix}
y_1+t\\
y_2+t\\
\vdots\\
y_n+t
\end{bmatrix}
\end{equation*}
are positive. Observe that $\widetilde{\vec{x}} \prec_{\w} \widetilde{\vec{y}}$ since adding $t$ to every coordinate adds $kt$ to the $k$th partial sum. 
Theorem \ref{T:UIN-weak-major-char-1} provides an $A \in \DSS_n$ such that $\widetilde{\vec{x}} = A \widetilde{\vec{y}}$. Then \eqref{E:UIN-char-sub-double} yields a $B \in \DS_n$ such that $A \leq_{\e} B$. Thus, $\widetilde{\vec{x}} = A \widetilde{\vec{y}}$ implies $\widetilde{\vec{x}} \leq_{\e} B \widetilde{\vec{y}}$, so $\vec{x} \leq_{\e} B \vec{y}$ because $B \vec{1} = \vec{1}$.
The result follows by taking $\vec{z}=B \vec{y}$.  Since $B \in \DS_n$, the Hardy--Littlewood--P\'olya theorem ensures that $B \vec{y} \prec \vec{y}$.

\medskip\noindent$(\Leftarrow$) If $\vec{x} \leq_{\e} \vec{z}$ and $\vec{z} \prec \vec{y}$, then $\vec{x} \prec_{\w} \vec{y}$. 
\end{proof}

For the application discussed in \S\ref{S:UIN-further-applications}, we need to know how majorization and weak majorization behave when a function is applied to the vectors involved. Given $\vec{x}=[x_i] \in \R^n$, assume that $f$ is a function defined on an open interval large enough to contain each entry of $\vec{x}$ and define 
\begin{equation*}
f(\vec{x}) := (f(x_1),f(x_2),\ldots,f(x_n)).
\end{equation*}

\begin{corollary} \label{C:UIN-f-convex-majorization-1}
Let $\vec{x},\vec{y} \in \R^n$, and let $I$ be an open interval in $\R$ containing the entries of $\vec{x}$ and $\vec{y}$. If $f$ is convex on $I$, then
\begin{equation*}
\vec{x} \prec \vec{y} 
\quad\implies\quad
f(\vec{x}) \prec_{\w} f(\vec{y}).
\end{equation*}
\end{corollary} 

\begin{proof}
Suppose that $\vec{x} \prec \vec{y}$. Theorem \ref{T:UIN-Hardy-Littlewood-Polya} provides
an $A=[a_{ij}] \in \DS_n$ such that $\vec{x} = A\vec{y}$. Thus,
\begin{equation*}
x_i = \sum_{j=1}^{n} a_{ij}y_j
\end{equation*}
for all $1 \leq i \leq n$.
Since $f$ is convex and $a_{i1},a_{i2},\ldots,a_{in} \geq 0$ sum to $1$, we have
\begin{equation*}
f(x_i) \leq \sum_{j=1}^{n} a_{ij} f(y_j)
\end{equation*}
for all $1 \leq i \leq n$.
We can rewrite the system of inequalities above as
\begin{equation*}
f(\vec{x}) \leq_{\e} A f(\vec{y}).
\end{equation*}
The Hardy--Littlewood--P\'{o}lya theorem says that $Af(\vec{y}) \prec f(\vec{y})$, and then, by Theorem \ref{T:UIN-weak-major-char-2}, we conclude that $f(\vec{x}) \prec_{\w} f(\vec{y})$.
\end{proof}

In the result above, the assumption $\vec{x} \prec \vec{y}$ cannot be replaced by $\vec{x} \prec_{\w} \vec{y}$. For example, $f(t)=|t|$ is convex and hence
\begin{equation*}
\vec{x} \prec \vec{y}
\quad\implies\quad
|\vec{x}| \prec_{\w} |\vec{y}|.
\end{equation*}
However, $(-1,-1) \prec_{\w} (0,0)$ but $f$ does not respect this relation since $(1,1) \nprec_w (0,0)$. 
An extra assumption is needed to allow such a conclusion.

\begin{corollary} \label{C:UIN-f-convex-majorization-2}
Let $\vec{x},\vec{y} \in \R^n$, and let $I$ be an open interval in $\R$ containing the entries of $\vec{x}$ and $\vec{y}$. If $f$ is increasing and convex on $I$, then
\begin{equation*}
\vec{x} \prec_{\w} \vec{y}
\quad\implies\quad
f(\vec{x}) \prec_{\w} f(\vec{y}).
\end{equation*}
\end{corollary}

\begin{proof}
Suppose that $\vec{x} \prec_{\w} \vec{y}$.  Theorem \ref{T:UIN-weak-major-char-2} provides a $\vec{z} \in \R^n$ such that $\vec{x} \leq_{\e} \vec{z}$ and $\vec{z} \prec \vec{y}$. Thus, the entries of $\vec{z}$ belong to $I$. Indeed, Rad\'o's theorem expresses $\vec{z}$ as a convex combination of permutations of $\vec{y}$; hence, each $z_i \in [\min_i y_i, \max_i y_i]$ and this interval lies in $I$. Since $f$ is increasing, $f(\vec{x}) \leq_{\e} f(\vec{z})$. By Corollary \ref{C:UIN-f-convex-majorization-1}, we also have $f(\vec{z}) \prec_{\w} f(\vec{y})$. Therefore, we conclude that $f(\vec{x}) \prec_{\w} f(\vec{y})$.
\end{proof}

If $\vec{x},\vec{y} \in \R_{+}^n$ and $\log \vec{x} \prec_{\w} \log \vec{y}$ (with the value $-\infty$ allowed), we say that $\vec{x}$ is \emph{weakly log-majorized} by $\vec{y}$, denoted $\vec{x} \prec_{\mathrm{wlog}} \vec{y}$. More explicitly, using the convention \eqref{E:UIN-decreasing-order-x}, $\log \vec{x} \prec_{\w} \log \vec{y}$ means that
\begin{equation*}
\prod_{i=1}^{k} x_{[i]} \leq \prod_{i=1}^{k} y_{[i]}
\end{equation*}
for all $1 \leq k \leq n$.  If equality holds with $k=n$, then
$\vec{x}$ is \emph{log-majorized} by $\vec{y}$, denoted $\vec{x} \prec_{\log} \vec{y}$.

\begin{corollary} \label{C:UIN-f-convex-majorization-3}
Let $\vec{x},\vec{y} \in \R^n_{+}$. Then
\begin{equation*}
\vec{x} \prec_{\mathrm{wlog}} \vec{y}
\quad\implies\quad 
\vec{x} \prec_{\w} \vec{y}.
\end{equation*}
\end{corollary}

\begin{proof}
First, suppose that every entry of $\vec{x}$ and $\vec{y}$ is positive. The assumption $\vec{x} \prec_{\mathrm{wlog}} \vec{y}$ means that $\log \vec{x} \prec_{\w} \log \vec{y}$. Apply Corollary \ref{C:UIN-f-convex-majorization-2} with $f(t)=e^t$ to conclude that  $\vec{x} \prec_{\w} \vec{y}$. Now suppose that some entries are zero.
If $p$ is the number of positive components of $\vec{x}$, then the product inequalities force the first $p$ components of $\vec{y}\!\!\downarrow$ to be positive.  Then apply the positive case to the length-$p$ vectors consisting of the positive components.  For $k>p$, 
the left partial sums remain constant, while the right partial sums are nondecreasing.
\end{proof}

\section{Monotonicity} \label{S:UIN-monotonicity}
We are at the point of presenting one of the major ingredients needed in the proof of von Neumann's theorem. This concerns the monotonicity of symmetric gauge functions. Recall the definition of these functions and their properties from \S\ref{S:UIN-absolute-norms}.

\begin{theorem} \label{T:UIN-monotonicity}
Let $\Phi$ be a symmetric gauge function on $\R^n$. If $\vec{x},\vec{y} \in \R^n_{+}$ and $\vec{x} \prec_{\w} \vec{y}$, then $\Phi(\vec{x}) \leq \Phi(\vec{y})$.
\end{theorem}

\begin{proof}
The proof of Theorem \ref{T:UIN-weak-major-char-2} provides a $\vec{z} \in \R^n_{+}$ such that $\vec{x} \leq_{\e} \vec{z}$ and $\vec{z} \prec \vec{y}$. By Rad\'o's theorem (Theorem \ref{T:UIN-Rado}), there are  $0 \leq t_{\sigma} \leq 1$ with $\sum_{\sigma \in \mathfrak{S}_n} t_{\sigma} =1$ such that
\begin{equation}\label{E:UIN-monotonicity1}
\vec{z} = \sum_{\sigma \in \mathfrak{S}_n} t_{\sigma} \vec{y}_\sigma.
\end{equation}
    Since a symmetric gauge function is monotone and the entries of $\vec{x}$ and $\vec{z}$ are nonnegative, the relation $\vec{x} \leq_{\e} \vec{z}$ implies $\Phi(\vec{x}) \leq \Phi(\vec{z})$. Since a symmetric gauge function is a norm on $\R^n$, \eqref{E:UIN-monotonicity1} gives
\begin{equation*}
\Phi(\vec{z}) \leq \sum_{\sigma \in \mathfrak{S}_n} t_{\sigma} \Phi(\vec{y}_\sigma).
\end{equation*}
Finally, a symmetric gauge function is permutation invariant, so $\Phi(\vec{y}_\sigma)=\Phi(\vec{y})$. Putting this together, we conclude that
\begin{equation*}
\Phi(\vec{x}) \leq \Phi(\vec{z}) \leq \sum_{\sigma \in \mathfrak{S}_n} t_{\sigma} \Phi(\vec{y}_\sigma) = \sum_{\sigma \in \mathfrak{S}_n} t_{\sigma} \Phi(\vec{y}) = \Phi(\vec{y}). \qedhere
\end{equation*}
\end{proof}

\section{The Courant--Fischer theorem} \label{S:UIN-the-min-max-thm}

Let $\h_n$ denote the set of all $n \times n$ Hermitian matrices.  Since
the eigenvalues of each $A \in \h_n$ are real, we write them in decreasing order
\begin{equation*}
\lambda_1(A) \geq \lambda_2(A) \geq \cdots \geq \lambda_n(A).
\end{equation*}
The following result of Courant--Fischer was first obtained by Ernst Fischer in 1905 \cite{MR1547416}, and later in 1920 by Richard Courant \cite{MR1544417}. However, the theorem became widely known through the text \cite{MR65391, MR140802}.

\begin{theorem}[Courant--Fischer \cite{MR1547416, MR1544417}] \label{T:max-min}
Let $A \in \h_n$. For any subspace $\V \subset \C^n$ of dimension $k$, 
\begin{equation*}
\lambda_k(A) 
\geq \min_{\substack{\vec{x} \in \V \\ \|\vec{x}\|=1}}  
\langle A \vec{x},\vec{x} \rangle.
\end{equation*}
Moreover, there is a subspace $\V_0 \subset \C^n$ of dimension $k$ such that
\begin{equation*}
\lambda_k(A) = \min_{\substack{\vec{x} \in V_0 \\ \|\vec{x}\|=1}} \langle A \vec{x},\vec{x} \rangle.
\end{equation*}
\end{theorem}

\begin{proof}
Let $\V \subseteq \C^n$ be a subspace with $\dim \V = k$. 
Let $\vec{x}_1,\vec{x}_2,\ldots,\vec{x}_n$ be orthonormal eigenvectors of $A$ corresponding to eigenvalues $\lambda_1,\lambda_2,\ldots,\lambda_n$, respectively. Let
\begin{equation*}
\W = \Span\{\vec{x}_k,\vec{x}_{k+1},\ldots,\vec{x}_n\}.
\end{equation*}
Since
\begin{align*}
\dim (\V \cap \W) 
&= \dim \V + \dim \W - \dim (\V+\W)\\
&= k + (n-k+1) - \dim (\V+\W)\\
&= 1+n - \dim (\V+\W) \geq 1,
\end{align*}
we have $\V \cap \W \neq \{ \vec{0}\}$. Each
unit vector $\vec{x} \in \V \cap \W$ can be written as
\begin{equation*}
\vec{x} = \alpha_{k} \vec{x}_{k} + \alpha_{k+1} \vec{x}_{k+1} + \cdots + \alpha_{n} \vec{x}_n,
\end{equation*}
in which $|\alpha_k|^2+|\alpha_{k+1}|^2+\cdots+|\alpha_n|^2=1$. Then
\begin{equation*}
A\vec{x} = \alpha_{k} \lambda_{k} \vec{x}_{k} + \alpha_{k+1} \lambda_{k+1} \vec{x}_{k+1} + \cdots + \alpha_{n} \lambda_{n} \vec{x}_n,
\end{equation*}
so
\begin{align*}
\langle A \vec{x},\vec{x} \rangle
&= |\alpha_{k}|^2 \lambda_{k} + |\alpha_{k+1}|^2 \lambda_{k+1} + \cdots + |\alpha_{n}|^2 \lambda_{n}\\
&\leq \left( |\alpha_{k}|^2  + |\alpha_{k+1}|^2  + \cdots + |\alpha_{n}|^2 \right) \lambda_{k} = \lambda_{k}.
\end{align*}
Therefore, 
\begin{equation*}
\lambda_k(A) \geq \min_{\substack{\vec{x} \in \V \\ \|\vec{x}\|=1}}
\langle A\vec{x},\vec{x} \rangle
\end{equation*}
for any subspace $\V$ of dimension $k$.

Consider
\begin{equation*}
\V_0 = \Span\{\vec{x}_1,\vec{x}_{2},\ldots,\vec{x}_k\}.
\end{equation*}
Let $\vec{x}=\vec{x}_k$ and deduce that
\begin{equation*}
\langle A \vec{x}_k,\vec{x}_k \rangle = \lambda_k(A).
\end{equation*}
For any other unit vector $\vec{x} \in \V_0$, we can write
\begin{equation*}
\vec{x} = \alpha_{1} \vec{x}_{1} + \alpha_{2} \vec{x}_{2} + \cdots + \alpha_{k} \vec{x}_k,
\end{equation*}
in which $|\alpha_1|^2+|\alpha_{2}|^2+\cdots+|\alpha_k|^2=1$. Since
\begin{equation*}
A\vec{x} = \alpha_{1} \lambda_{1} \vec{x}_{1} + \alpha_{2} \lambda_{2} \vec{x}_{2} + \cdots + \alpha_{k} \lambda_{k} \vec{x}_k,
\end{equation*}
we conclude that
\begin{align*}
\langle A\vec{x},\vec{x} \rangle
&= |\alpha_{1}|^2 \lambda_{1} + |\alpha_{2}|^2 \lambda_{2} + \cdots + |\alpha_{k}|^2 \lambda_{k}\\
&\geq \left( |\alpha_{1}|^2  + |\alpha_{2}|^2  + \cdots + |\alpha_{k}|^2 \right) \lambda_{k} = \lambda_{k}.
\end{align*}
Therefore,
\begin{equation*}
\lambda_k(A) = \min_{\substack{\vec{x} \in \V_0 \\ \|\vec{x}\|=1}}  \langle A\vec{x},\vec{x} \rangle. \qedhere
\end{equation*} 
\end{proof}

The Courant--Fischer theorem can be written as
\begin{equation}\label{E:UIN-max-min-formula-1}
\lambda_k(A) = \max_{\substack{\V \subset \C^n \\ \dim \V=k}}   \, \,
\min_{\substack{x \in \V \\ \|\vec{x}\|=1}}  \langle A\vec{x},\vec{x} \rangle.
\end{equation}
If we replace $A$ by $-A$, the relation above gives
\begin{equation}\label{E:UIN-max-min-formula-1b}
\lambda_k(A) = \min_{\substack{\V \subset \C^n \\ \dim \V=n-k+1}} 
\max_{\substack{\vec{x} \in \V \\ \|\vec{x}\|=1}}  \langle A\vec{x},\vec{x} \rangle.
\end{equation}
Consequently, Theorem \ref{T:max-min} is also called the max-min theorem or min-max theorem. 

Recall that the singular values of $A \in \M_n$ are nonnegative square roots of the eigenvalues of $A^*A$.  Apply Theorem \ref{T:max-min} and take square roots to see that for any subspace $\V \subset \C^n$ of dimension $k$, 
\begin{equation*}
s_k(A) \geq \min_{\substack{\vec{x} \in \V \\ \|\vec{x}\|=1}}  \|A\vec{x}\|
\end{equation*}
and there is a subspace $\V_0 \subset \C^n$ of dimension $k$ such that
\begin{equation*}
s_k(A) = \min_{\substack{\vec{x} \in \V_0 \\ \|\vec{x}\|=1}}  \|A\vec{x}\|.
\end{equation*}
As in \eqref{E:UIN-max-min-formula-1} and \eqref{E:UIN-max-min-formula-1b}, we may also write
\begin{align}
s_k(A) 
&= \max_{\substack{\V \subset \C^n \\ \dim \V=k}} \,\, \min_{\substack{ \vec{x} \in \V \\ \|\vec{x}\|=1}}  \|Ax\| \notag\\
&= \min_{\substack{\V \subset \C^n \\ \dim \V=n-k+1}} \,\, \max_{\substack{\vec{x} \in \V \\ \|\vec{x}\|=1}} \,\, \|A\vec{x}\|. \label{E:UIN-max-min-formula-2}
\end{align}
As a special case of \eqref{E:UIN-max-min-formula-2}, we see that
\begin{equation}\label{E:UIN-s1=norm-1}
s_1(A) = \|A\|_{\mathrm{op}},
\end{equation}
which we observed in \eqref{E:UIN-spectral-norm-s1}. We can use compound matrices to get more results from this special case. Let us first recall the definition and properties of compound matrices \cite{MR2978290, MR3076701}. 

Let $A \in \M_n$ and fix $1 \leq k \leq n$. Then for
\begin{equation*}
1 \leq i_1 < i_2 < \cdots < i_k \leq n 
\quad\text{and}\quad 
1 \leq j_1 < j_2 < \cdots < j_k \leq n, 
\end{equation*}
we denote the $k \times k$ submatrix of $A$ formed with rows $i_1,i_2,\ldots,i_k$ and 
columns $j_1,j_2,\ldots,j_k$ of $A$ by 
\begin{equation}\label{eq:Submatrix}
A[i_1,i_2,\ldots,i_k\mid j_1,j_2,\ldots,j_k].
\end{equation}
Put a lexicographical order on 
\begin{equation*}
\mathcal{E}(k,n) := \{(\ell_1,\ell_2,\ldots,\ell_k) : 1 \leq \ell_1 < \ell_2 < \cdots < \ell_k \leq n \}.
\end{equation*}
This set has $\binom{k}{n}$ elements.  The \emph{compound matrix} $C_k(A)$ is a $\binom{k}{n} \times \binom{k}{n}$ matrix whose entries are
\begin{equation*}
\det A[i_1,i_2,\ldots,i_k \mid j_1,j_2,\ldots,j_k],
\end{equation*}
and the rows and columns are ordered as in $\mathcal{E}(k,n)$. The main properties of the compound matrix that we need are the following.
\begin{enumerate}[(a)]
\item The compound operation is multiplicative:
\begin{equation}\label{E:UIN-s1=norm-2a}
C_k(AB) = C_k(A)  C_k(B).
\end{equation}

\item If $\boldsymbol{\lambda}(A)=[\lambda_1~\lambda_2~\ldots~\lambda_n]^{\T}$ is the vector of eigenvalues of $A$, then
\begin{equation}\label{E:UIN-s1=norm-2b}
\boldsymbol{\lambda}\big( C_k(A) \big) =
\big\{ \lambda_{\ell_1}\lambda_{\ell_2}\cdots \lambda_{\ell_k} : 1 \leq \ell_1 < \ell_2 < \cdots < \ell_k \leq n \}.
\end{equation}

\item If $\vec{s}(A)= [s_1~s_2~\ldots~s_n]^{\T}$ is the vector of singular values of $A$ in decreasing order, then
\begin{equation}\label{E:UIN-s1=norm-2c}
\vec{s}\big( C_k(A) \big) =
\big\{ s_{\ell_1}s_{\ell_2}\cdots s_{\ell_k} : 1 \leq \ell_1 < \ell_2 < \cdots < \ell_k \leq n \big\},
\end{equation}
in which the singular values of $C_k(A)$ are in decreasing order.
\end{enumerate}

The set notation in \eqref{E:UIN-s1=norm-2b} and \eqref{E:UIN-s1=norm-2c} is somewhat misleading
since multiplicity is taken into account.  That is, $\boldsymbol{\lambda}\big( C_k(A) \big)$ and $\vec{s}\big( C_k(A) \big)$ are multisets. 

As a special case of \eqref{E:UIN-s1=norm-2c}, we have
\begin{equation}\label{E:UIN-s1=norm-3}
s_1\big( C_k(A) \big) = s_1s_2\cdots s_k.
\end{equation}

Given $\vec{x}=[x_1~x_2~\ldots~x_n]^{\T}$ and $\vec{y}=[y_1~y_2~\ldots~y_n]^{\T}$, their Hadamard--Schur product is
\begin{equation*}
\vec{x}*\vec{y} := [x_1y_1 \,\, x_2y_2\,\, \ldots, \,\, x_ny_n]^{\T}.
\end{equation*}

\begin{corollary}[A.~Horn \cite{MR45316}]
If $A,B \in \M_n$, then
\begin{equation} \label{C:UIN-s1=norm-4}
s(AB) \prec_{\log} s(A)*s(B).
\end{equation}
\end{corollary}

\begin{proof}
Fix $1 \leq k \leq n$. Then
\begin{align*}
\prod_{i=1}^{k} s_i(AB) 
&= s_1\big( C_k(AB) \big) && \text{by  \eqref{E:UIN-s1=norm-3}}\\
&= \|C_k(AB)\|_{\mathrm{op}} && \text{by  \eqref{E:UIN-s1=norm-1}}\\
&= \|C_k(A) \, C_k(B)\|_{\mathrm{op}} && \text{by  \eqref{E:UIN-s1=norm-2a}}\\
&\leq  \|C_k(A)\|_{\mathrm{op}} \, \|C_k(B)\|_{\mathrm{op}} && \text{by  \eqref{E:UIN-spectral-norm-s-mul}}\\
&=  s_1\big( C_k(A) \big) \, s_1\big( C_k(B) \big) && \text{by  \eqref{E:UIN-s1=norm-1}}\\
&=  \prod_{j=1}^{k} s_j(A) \, \prod_{j=1}^{k} s_j(B) && \text{by  \eqref{E:UIN-s1=norm-3}}\\
&= \prod_{j=1}^{k} s_j(A) \,  s_j(B). 
\end{align*}
Equality holds if $k=n$ since
\begin{align*}
\prod_{i=1}^{n} s_i(AB) 
&= |\det (AB)| && \text{by \eqref{E:UIN-A=udiagV-0}}\\
&=|\det (A)| \,\, |\det (B)| \\
&= \prod_{j=1}^{n} s_j(A) \, \prod_{j=1}^{n} s_j(B) \\
&= \prod_{j=1}^{n} s_j(A) s_j(B). 
\end{align*}
Therefore, $s(AB) \prec_{\log} s(A)*s(B)$.
\end{proof}

If we apply Corollary \ref{C:UIN-f-convex-majorization-3} to \eqref{C:UIN-s1=norm-4}, then we see that for each $A,B \in \M_n$,
\begin{equation}\label{C:UIN-s1=norm-5}
s(AB) \prec_{\w} s(A)*s(B).
\end{equation}

We next use compound matrices to obtain a result on weak majorization in which eigenvalues of $A$ are involved.  Suppose that $A \in \M_n$; that is, $A$ is not necessarily Hermitian.
Let $\lambda_1(A),\lambda_2(A),\ldots,\lambda_n(A)$ denote the eigenvalues of $A$,
repeated according to their multiplicities and ordered so that
\begin{equation*}
|\lambda_1(A)| \geq |\lambda_2(A)| \geq \cdots \geq |\lambda_n(A)|.
\end{equation*}

\begin{corollary}[Weyl \cite{MR30693}]
Let $A \in \M_n$. Then
\begin{equation} \label{L:UIN-s1=norm-6}
|[\lambda_1(A)~\lambda_2(A)~\ldots~\lambda_n(A)]^{\T}| \prec_{\log} s(A).  
\end{equation}
\end{corollary}

\begin{proof}
Let $\rho(\cdot)$ denote the \emph{spectral radius} of the indicated matrix; that is, the maximum modulus attained by an eigenvalue of the matrix.
For a fixed $1 \leq k \leq n$,
\begin{align*}
\prod_{i=1}^{k} |\lambda_i(A)| &= \rho\big( C_k(A) \big) && \text{by \eqref{E:UIN-s1=norm-2b}}\\
&\leq \|C_k(A)\|_{\mathrm{op}} \\
&= s_1\big( C_k(A) \big) && \text{by \eqref{E:UIN-s1=norm-1}}\\
&= \prod_{j=1}^{k} s_j(A) && \text{by \eqref{E:UIN-s1=norm-3}}.
\end{align*}
Equality holds for $k=n$ since
\begin{equation*}
\prod_{i=1}^{n} |\lambda_i(A)| = |\det A| = \prod_{i=1}^{n} s_i(A). \qedhere
\end{equation*}
\end{proof}

If we apply Corollary \ref{C:UIN-f-convex-majorization-3} to \eqref{L:UIN-s1=norm-6}, then
we see that for each $A \in \M_n$,
\begin{equation} \label{C:UIN-s1=norm-7}
|[\lambda_1(A)~\lambda_2(A)~\ldots~\lambda_n(A)]^{\T}| \prec_{\w} \vec{s}(A).
\end{equation}
In particular, 
\begin{equation} \label{C:UIN-s1=norm-7b}
|\tr A| \leq \sum_{j=1}^{n} s_j(A).
\end{equation}

\section{The Cauchy interlacing theorem and Weyl monotonicity principle} \label{UIN-Cauchy}
Given $A \in \h_n$ and $\{i_1,i_2,\ldots,i_m\} \subseteq \{1,2,\ldots,n\}$ with cardinality $m$, 
\begin{equation*}
B = A[i_1,i_2,\ldots,i_m \mid i_1,i_2,\ldots,i_m] \in \M_m
\end{equation*}
is a \emph{principal submatrix} of $A$. Note that $B$ is also Hermitian and hence its eigenvalues are real. 
Thus, we may ask about the relationship between the eigenvalues of $A$ and those of $B$.  
In 1829, Cauchy established interlacing results for the eigenvalues of principal submatrices of Hermitian matrices \cite{Cauchy-1,MR2866794}. A modern reference is \cite[Thm.~4.3.28]{MR2978290}.
This result is also known as the Poincar\'{e} separation theorem.  

\begin{theorem}[Cauchy \cite{Cauchy-1, MR2866794}] \label{T:UIN-cauchy}
If $A \in \h_n$ and $B \in \h_m$ is a principal submatrix of $A$, then 
\begin{equation*}
\lambda_j(A) \geq \lambda_j(B)  \geq \lambda_{j+n-m}(A)
\end{equation*}
for all $1 \leq j \leq m$.
\end{theorem}

\begin{proof}
If $m=n$, there is nothing to prove, so assume that $m <n$.
Since permutation similarity preserves eigenvalues, we may assume that
\begin{equation*}
A =
\begin{bmatrix}
B & C \\
C^* & D
\end{bmatrix}.
\end{equation*}
The Courant--Fischer theorem (Theorem \ref{T:max-min}) provides a subspace $\V_0 \subset \C^m$ of dimension $j$ such that
\begin{equation}\label{E:UIN-cauchy-B}
\lambda_j(B) = \min_{\substack{\vec{x} \in \V_0 \\ \|\vec{x}\|=1}} \langle B\vec{x},\vec{x} \rangle_{\C^m}.
\end{equation}
We embed $\V_0$ in $\C^n$ as follows. For each $\vec{x} \in \C^m$, let
\begin{equation*}
\widehat{\vec{x}} =
\begin{bmatrix}
\vec{x} \\
\vec{0}
\end{bmatrix} \in \C^n,
\end{equation*}
in which $\vec{0}$ has length $n-m$, and
\begin{equation*}
\widehat{\V}_0 = \{ \widehat{\vec{x}} : \vec{x} \in \V_0 \}.
\end{equation*}
Then for all $\vec{x} \in \C^m$, 
\begin{equation}\label{E:UIN-cauchy-BA}
\langle B\vec{x},\vec{x} \rangle_{\C^m} = \langle A \widehat{\vec{x}},\widehat{\vec{x}} \rangle_{\C^n}.
\end{equation}
Another application of the Courant--Fischer theorem yields
\begin{equation*}
\lambda_j(A) \geq \min_{\substack{\widehat{\vec{x}} \in \widehat{\V}_0 \\ \widehat{\|\vec{x}}\|=1}}  \langle A\widehat{\vec{x}},\widehat{\vec{x}} \rangle_{\C^n}.
\end{equation*}
Thus, \eqref{E:UIN-cauchy-B} and \eqref{E:UIN-cauchy-BA} ensure that
\begin{equation*}
\lambda_j(A) \geq \min_{\substack{\vec{x} \in \V_0 \\ \|\vec{x}\|=1}}  \langle B\vec{x},\vec{x} \rangle_{\C^m} = \lambda_j(B). 
\end{equation*}
For the second desired inequality, apply the first inequality to $-A$ and $-B$ with index $m-j+1$.
\end{proof}

\begin{example}\label{Example:Interlacing}
The eigenvalues of 
    \begin{equation*}
        A=
        \begin{bmatrix*}[r]
             4 & -1 & 2 & 3 & -1 \\
             -1 & 8 & 1 & -2 & 2 \\
             2 & 1 & 6 & 1 & -3 \\
             3 & -2 & 1 & -3 & -2 \\
             -1 & 2 & -3 & -2 & -5 \\
        \end{bmatrix*} 
    \end{equation*}
are (rounded to two decimal places) $9.58, 8.13, 2.84, -3.81,  -6.75$.
The eigenvalues of its $4 \times 4$ principal submatrix
\begin{equation*}
    B = 
\begin{bmatrix*}[r]
 4 & -1 & 2 & 3 \\
 -1 & 8 & 1 & -2 \\
 2 & 1 & 6 & 1 \\
 3 & -2 & 1 & -3 \\
\end{bmatrix*}
\end{equation*}
are (rounded to two decimal places) $8.89, 7.76, 2.68,  -4.32$.
As predicted by Theorem \ref{T:UIN-cauchy}, the eigenvalues of $A$ and $B$ interlace; see Figure \ref{Figure:SubmatrixInterlacing}.
\end{example}

\begin{figure}\centering
    \begin{tikzpicture}[xscale=0.5, yscale=0.5]
        \pgfmathsetmacro{\startl}{-6}
        \pgfmathsetmacro{\endl}{12}
        \def\aone{4}
        \def\atwo{3.763932023, 8.236067977}
        \def\athree{2.422910555, 7.153467305, 8.423622140}
        \def\afour{-4.322185309, 2.678158422, 7.758823368, 8.885203519}
        \def\afive{-6.745725045, -3.808065603, 2.841117710, 8.130909414, 9.581763524}
        \pgfmathsetmacro{\xx}{\startl+2}
        \pgfmathsetmacro{\startll}{\startl-.9}
        \pgfmathsetmacro{\endll}{\endl+.9}
        \foreach\y in {0,1.5,3,4.5,6}{
            \draw (\startll,\y)--(\endll,\y);
            \foreach\x in {\startl,\xx,...,\endl}{
                \pgfmathsetmacro{\upy}{\y + .1}
                \pgfmathsetmacro{\downy}{\y - .1}
                \draw[] (\x,\downy)--(\x,\upy);
            }
        }

        \foreach\x in {\startl,\xx,...,\endl}{
            \node[below] at (\x,-0.2){\pgfmathprintnumber{\x}};
        }

        \pgfmathsetmacro{\radius}{.15}
        \foreach\x in \aone{
            \draw[blue] (\x,0)--(\x,6);
            \fill[blue] (\x,0) circle (\radius);
            \draw[black] (\x,0) circle (\radius);
        }
        \foreach\x in \atwo{
            \draw[red] (\x,1.5)--(\x,6);
            \fill[red] (\x,1.5) circle (\radius);
            \draw[black] (\x,1.5) circle (\radius);
        }
        \foreach\x in \athree{
            \draw[cyan] (\x,3)--(\x,6);
            \fill[cyan] (\x,3) circle (\radius);
            \draw[black] (\x,3) circle (\radius);
        }
        \foreach\x in \afour{
            \draw[orange] (\x,4.5)--(\x,6);
            \fill[orange] (\x,4.5) circle (\radius);
            \draw[black] (\x,4.5) circle (\radius);
        }
        \foreach\x in \afive{
            \fill[yellow] (\x,6) circle (\radius);
            \draw[black] (\x,6) circle (\radius);
        }
        \pgfmathsetmacro{\xx}{\endll+1.0}
        \node at (\xx+0.5,0) {$n=1$};
        \node at (\xx+0.5,1.5) {$n=2$};
        \node at (\xx+0.5,3) {$n= 3$};
        \node at (\xx+0.5,4.5) {$n=4$};
        \node at (\xx+0.5,6) {$n=5$};
    \end{tikzpicture}
    \caption{Eigenvalues of the $n \times n$ leading principal submatrices of 
    $A$ in Example \ref{Example:Interlacing}.}
    \label{Figure:SubmatrixInterlacing}
\end{figure}
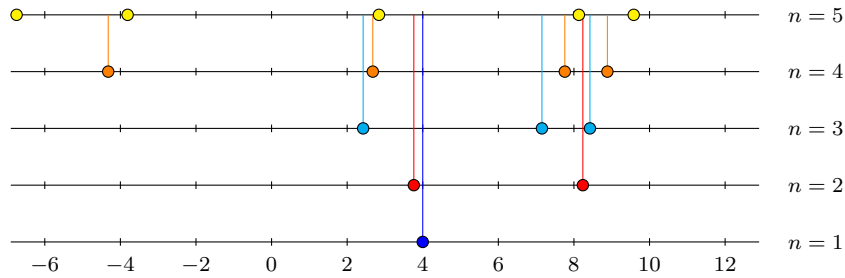

Here is another important application of the Courant--Fischer theorem \cite[Thm.~4.3.1]{MR2978290}.

\begin{theorem}[Weyl Monotonicity Principle] \label{T:UIN-Cauchy-Weyl-1}
Let $A,B \in \h_n$. Assume that $1 \leq r,s \leq n$, with $r+s \geq n+1$. Then
\begin{equation*}
\lambda_r(A) + \lambda_s(B) \leq \lambda_{r+s-n}(A+B).
\end{equation*}
In particular, if $A \leq B$, then
\begin{equation*}
\lambda_r(A) \leq \lambda_{r}(B), \qquad 1 \leq r \leq n.
\end{equation*}
\end{theorem}

\begin{proof}
The Courant--Fischer theorem (Theorem \ref{T:max-min}) provides subspaces $\V_0,\W_0 \subset \C^n$ of dimension $r$ and $s$, respectively, such that 
\begin{equation*}
\lambda_r(A) = \min_{\substack{\vec{x} \in \V_0 \\ \|\vec{x}\|=1}}  \langle A\vec{x},\vec{x} \rangle,
\quad\text{and}\quad
\lambda_s(B) = \min_{\substack{\vec{x} \in \W_0 \\ \|\vec{x}\|=1}}  \langle B \vec{x},\vec{x} \rangle.
\end{equation*}
Since
\begin{equation*}
\dim (\V_0 \cap \W_0) 
=
\dim (\V_0) + \dim (\W_0) - \dim (\V_0 + \W_0)
\geq r+s-n,
\end{equation*}
there is a subspace $\U_0 \subset \V_0 \cap \W_0$ of dimension $r+s-n$. 
Then the Courant--Fischer theorem yields
\begin{align*}
\lambda_{r+s-n}(A+B) 
&\geq \min_{\substack{\vec{x} \in \U_0 \\ \|\vec{x}\|=1}}  \langle (A+B)\vec{x},\vec{x} \rangle\\
&= \min_{\substack{\vec{x} \in \U_0 \\ \|\vec{x}\|=1}}  \big( \langle A\vec{x},\vec{x} \rangle + \langle B\vec{x},\vec{x} \rangle \big)\\
&\geq \min_{\substack{\vec{x} \in \U_0 \\ \|\vec{x}\|=1}}   \langle A\vec{x},\vec{x} \rangle  
+ \min_{\substack{\vec{x} \in \U_0 \\ \|\vec{x}\|=1}}  \langle B\vec{x},\vec{x} \rangle \\
&\geq \min_{\substack{\vec{x} \in \V_0 \\ \|\vec{x}\|=1}} \,\,  \langle A\vec{x},\vec{x} \rangle  + \min_{\substack{\vec{x} \in \W_0 \\ \|\vec{x}\|=1}}   \langle B\vec{x},\vec{x} \rangle \\
&= \lambda_r(A) + \lambda_s(B).
\end{align*}
In particular, 
\begin{equation*}
\lambda_r(X) + \lambda_n(Y) \leq \lambda_{r}(X+Y)
\end{equation*}
for any $X,Y \in \h_n$ and $1 \leq r \leq n$. Given $A,B \in \h_n$ with $A \leq B$, let $X=A$ and $Y=B-A \geq 0$ to conclude $\lambda_r(A) \leq \lambda_{r}(B)$. 
\end{proof}

\section{The Ky Fan majorization theorem} \label{S:UIN-fan}
There are various inequalities due to Ky Fan that involve eigenvalues, norms, or determinants. The following one relates the eigenvalues of $A$, $B$, and $A+B$. Recall that we always arrange the eigenvalues of Hermitian matrices in decreasing order.

\begin{theorem}[Ky Fan \cite{MR34519, MR45952}] \label{T:UIN-fan}
If $A,B \in \h_n$, then
\begin{equation*}
\boldsymbol{\lambda}(A+B) \prec \boldsymbol{\lambda}(A) + \boldsymbol{\lambda}(B).
\end{equation*}
\end{theorem}

\begin{proof}
The proof is based on an identity that is of independent interest. We show that for each $X \in \h_n$ 
and each fixed $1 \leq k \leq n$,
\begin{equation}\label{E:UIN-fan1}
\sum_{j=1}^{k} \lambda_j(X) = \max_{U^*U=I_k} \tr (U^*XU),
\end{equation}
in which $U \in \M_{n\times k}$. Let $V = [U~U'] \in \M_n$ be unitary; if $k = n$, then $U'$ is omitted and $V = U$.
Then $U^*XU$ is a principal submatrix of $V^*XV$, so the Cauchy interlacing theorem (Theorem \ref{T:UIN-cauchy}) ensures that
\begin{equation*}
\sum_{j=1}^{k} \lambda_j(X) =
\sum_{j=1}^{k} \lambda_j(V^*XV) \geq
\sum_{j=1}^{k} \lambda_j(U^*XU) =
\tr (U^*XU),
\end{equation*}
in which $U^*XU \in \M_k$. Conversely, the spectral theorem provides a unitary $W \in \M_n$ such that
\begin{equation*}
W^*XW = \diag \big(\lambda_1(X),\lambda_2(X),\ldots,\lambda_n(X)\big).
\end{equation*}
Let $U_0 \in \M_{n \times k}$ be the $n \times k$ matrix consisting of the first $k$ columns of $W$. Then $U_0^*U_0=I_k$ and
\begin{equation*}
U_0^*XU_0 = \diag \big(\lambda_1(X),\lambda_2(X),\ldots,\lambda_k(X)\big).
\end{equation*}
Thus,
\begin{equation*}
\tr (U_0^*XU_0) = \sum_{j=1}^{k} \lambda_j(X),
\end{equation*}
which establishes \eqref{E:UIN-fan1}.

For $1 \leq k \leq n$, \eqref{E:UIN-fan1} ensures that
\begin{align*}
\sum_{j=1}^{k} \lambda_j(A+B) 
&= \max_{U^*U=I_k} \tr (U^*(A+B)U) \\
&= \max_{U^*U=I_k} \tr (U^*AU+U^*BU) \\
&= \max_{U^*U=I_k} \big(\tr (U^*AU)+\tr (U^*BU)\big) \\
&\leq \max_{U^*U=I_k} \tr (U^*AU)+ \max_{U^*U=I_k}\tr (U^*BU) \\
&= \sum_{j=1}^{k} \lambda_j(A) + \sum_{j=1}^{k} \lambda_j(B).
\end{align*}
If $k=n$, then
\begin{equation*}
\sum_{j=1}^{n} \lambda_j(A+B)
= \tr (A+B)
= \tr (A) + \tr (B)
= \sum_{j=1}^{n} \lambda_j(A)  + \sum_{j=1}^{n} \lambda_j(B),
\end{equation*}
from which we conclude that $\lambda(A+B) \prec \lambda(A) + \lambda(B)$.
\end{proof}

\section{The weak subadditivity of singular values} \label{S:UIN-weak-aubadd}

We want a version of Ky Fan's theorem \ref{T:UIN-fan} 
for the singular values of an arbitrary matrix. 
The following crucial result shows that the singular values are weakly subadditive.

\begin{theorem} \label{T:UIN-weak-aubadd}
If $A,B \in \M_n$, then
\begin{equation*}
s(A+B) \prec_{\w} s(A) + s(B).
\end{equation*}
\end{theorem}

\begin{proof}
To exploit Ky Fan's theorem (Theorem \ref{T:UIN-fan}), 
we relate the singular values of $X \in \M_n$ to 
the eigenvalues of 
\begin{equation}\label{E:def-gamma-x}
\Gamma(X) :=
\begin{bmatrix}
0 & X^* \\
X & 0
\end{bmatrix} \in \h_{2n}.  
\end{equation}
The singular value decomposition provides unitaries $U,V \in \M_n$ such that
\begin{equation*}
U^*XV = \diag (s_1(X),s_2(X),\ldots,s_n(X)).
\end{equation*}
Let
\begin{equation*}
W =
\frac{1}{\sqrt{2}}
\begin{bmatrix}
V & V \\
U & -U
\end{bmatrix} \in \M_{2n}
\end{equation*}
and verify that $W$ is unitary and
\begin{equation*}
W^*\Gamma(X)W = \diag (s_1(X),s_2(X),\ldots,s_n(X),-s_1(X),-s_2(X),\ldots,-s_n(X)).
\end{equation*}
Writing eigenvalues in decreasing order as always, we have
\begin{equation*}
\lambda\big(\Gamma(X)\big) = \big(s_1(X),s_2(X),\ldots,s_n(X),-s_n(X),-s_{n-1}(X),\ldots,-s_1(X)\big).
\end{equation*}
In particular, for $1 \leq j \leq n$, we have
\begin{equation*}
\lambda_j\big(\Gamma(X)\big) = s_j(X).
\end{equation*}
For $1 \leq k \leq n$, Ky Fan's theorem \ref{T:UIN-fan} implies that
\begin{align*}
\sum_{j=1}^{k} s_j(A+B)
&=  \sum_{j=1}^{k} \lambda_j\big(\Gamma(A+B)\big)\\
&= \sum_{j=1}^{k} \lambda_j\big(\Gamma(A)+\Gamma(B)\big)\\
&\leq \sum_{j=1}^{k} \lambda_j\big(\Gamma(A)\big) + \sum_{j=1}^{k} \lambda_j\big(\Gamma(B)\big)\\
&= \sum_{j=1}^{k} s_j(A) + \sum_{j=1}^{k} s_j(B)\\
&= \sum_{j=1}^{k} (s_j(A) + s_j(B)). \qedhere
\end{align*}
\end{proof}

\section{The von Neumann theorem} \label{S:UIN-on-neumann}
Suppose that $\|\cdot\|$ is a unitarily invariant norm on $\M_n$ and define $\Phi: \R^n \to \R_{+}$ by
\begin{equation*}
\Phi(\vec{x}) := \|\diag (x_1,x_2,\ldots,x_n)\|.
\end{equation*}
Then $\Phi$ is a monotone symmetric norm on $\R^n$; that is, it is a symmetric gauge function. In particular, 
\begin{equation*}
\Phi(\vec{x}) = \Phi(|\vec{x}|\!\!\downarrow)
\end{equation*}
for all $\vec{x} \in \R^n$.  
This observation and the singular value decomposition suggest that we start with a function $\phi$, initially defined on $\R^n_{+}\!\!\downarrow$, extend it to $\R^n$ via
\begin{equation}\label{E:UIN-extension-phi}
\Phi(\vec{x}) = \phi(|\vec{x}|\!\!\downarrow), \qquad \vec{x} \in \R^n
\end{equation}
so that monotonicity and permutation invariance are respected,
and then consider when its extension is a symmetric gauge function. 

In 1937, John von Neumann \cite{vonNeumann, MR157874} established that every unitarily invariant norm on $\M_n$ arises from a symmetric gauge function applied to the singular values of a matrix and that any symmetric gauge function provides a unitarily invariant norm on $\M_n$. See also Lemma \ref{L:UIN-Further-examples}.

\begin{theorem}[von Neumann \cite{vonNeumann, MR157874}] \label{T:UIN-von-Neumann}
Let $\phi: \R^n_{+}\!\!\downarrow \to \R_{+}$ be given and extend it to $\R^n$ via \eqref{E:UIN-extension-phi}. Define $\|\cdot\|_{\phi}$ on $\M_n$ by
\begin{equation}\label{E:UIN-von-neumann-def}
\|A\|_{\phi} := \phi\big( s_1(A),s_2(A),\ldots,s_n(A) \big).
\end{equation}
Then $\|\cdot\|_{\phi}$ is a unitarily invariant norm on $\M_n$ if and only if $\Phi$ is a symmetric gauge function on $\R^n$.
\end{theorem}

\begin{proof}
Suppose that $\|\cdot\|_{\phi}$ is a unitarily invariant norm on $\M_n$.  For each $\vec{x} \in \R^n$, we have $\vec{s}(\diag \vec{x}) = | \vec{x} |\!\!\downarrow$, so \eqref{E:UIN-von-neumann-def} ensures that
\begin{equation*}
    \Phi(\vec{x}) = \phi(|\vec{x}|\!\!\downarrow) = \| \diag \vec{x} \|_{\phi}.
\end{equation*}
Since the map $\vec{x} \mapsto \diag \vec{x}$ is linear and injective, the restriction of
$\| \cdot \|_{\phi}$ to the diagonal matrices shows that $\Phi$ is a norm on $\R^n$.
Moreover, the definition $\Phi(\vec{x}) = \phi(|\vec{x}|\!\!\downarrow)$ ensures that
$\Phi$ is absolute and permutation invariant.  By Lemma \ref{L:UIN-monotone=absolute},
every absolute norm is monotone.  Therefore, $\Phi$ is a symmetric gauge function.

Now suppose that $\Phi$ is a symmetric gauge function and define $\|\cdot\|_{\phi}$ by \eqref{E:UIN-von-neumann-def}. Then the triangle inequality is the only property of a norm that is not clearly fulfilled. To verify the triangle inequality, let $A,B \in \M_n$. Then the weak subadditivity theorem (Theorem \ref{T:UIN-weak-aubadd}) ensures that
\begin{equation*}
\vec{s}(A+B) \prec_{\w} \vec{s}(A) + \vec{s}(B).
\end{equation*}
Hence the monotonicity theorem (Theorem \ref{T:UIN-monotonicity}) yields
\begin{equation*}
\Phi\big(\vec{s}(A+B)\big) \leq \Phi\big(\vec{s}(A) + \vec{s}(B)\big).
\end{equation*}
Since $\Phi$ is a norm,
\begin{equation*}
\Phi\big(\vec{s}(A) + \vec{s}(B)\big) \leq \Phi\big(\vec{s}(A)\big) + \Phi\big(\vec{s}(B)\big).
\end{equation*}
According to the definition \eqref{E:UIN-von-neumann-def} of $\|\cdot\|_{\phi}$, we have
\begin{equation*}
\Phi\big(\vec{s}(A+B)\big) = \|A+B\|_{\phi},
\quad
\Phi\big(\vec{s}(A)\big) = \|A\|_{\phi},
\quad\text{and}\quad
\Phi\big(\vec{s}(B)\big) = \|B\|_{\phi}.
\end{equation*}
Putting this all together, we find that
\begin{equation*}
\|A+B\|_{\phi} \leq \|A\|_{\phi} + \|B\|_{\phi}.
\end{equation*}
Finally, note that if $B=UAV$, in which $U$ and $V$ are unitary, then $\vec{s}(A)=\vec{s}(B) \in \R^n_{+}\!\!\downarrow$ and thus, again by \eqref{E:UIN-von-neumann-def}, $\|A\|_{\phi} = \|B\|_{\phi}$. Therefore, $\|\cdot\|_{\phi}$ is a unitarily invariant norm.
\end{proof}

\section{Further examples} \label{S:UIN-further-examples}
A major application of Theorem \ref{T:UIN-von-Neumann} is to take some symmetric gauge functions on $\R^n$ and use them to build the corresponding unitarily invariant norms on $\M_n$. We provide some examples in this section. In order to facilitate verifying that $\Phi$ is a symmetric gauge function, we provide the following useful lemma.

\begin{lemma} \label{L:UIN-Further-examples}
Let $\phi: \R^n_{+}\!\!\downarrow\, \to \R_{+}$ and define its extension to $\R^n$ by \eqref{E:UIN-extension-phi}. Then $\Phi$ is a symmetric gauge function on $\R^n$ if and only if the following hold.
\begin{enumerate}[(a)]\addtolength{\itemsep}{3pt}
\item $\phi(\vec{x}) \geq 0$ for all $\vec{x} \in \R^n_{+}\!\!\downarrow$.
\item $\phi(\vec{x}) = 0$ if and only if $\vec{x}=0$. 
\item $\phi(a \vec{x}) = a\phi(\vec{x})$ for all $\vec{x} \in \R^n_{+}\!\!\downarrow$ and all $a \in \R_{+}$.
\item $\phi(\vec{x}+\vec{y}) \leq \phi(\vec{x}) + \phi(\vec{y})$ for all $\vec{x},\vec{y} \in \R^n_{+}\!\!\downarrow$.
\item $\phi(\vec{x}) \leq \phi(\vec{y})$ whenever $\vec{x},\vec{y} \in \R^n_{+}\!\!\downarrow$ and $\vec{x} \prec_{\w} \vec{y}$.
\end{enumerate}
\end{lemma}

\begin{proof}
$(\Rightarrow)$ Suppose that $\Phi$ is a symmetric gauge function. Then the first four properties are part of the definition of this concept, and the fifth is the content of the monotonicity theorem (Theorem \ref{T:UIN-monotonicity}). Note that $\R^n_{+}\!\!\downarrow$ is closed under addition and multiplication by scalars $a \in \R_{+}$.

\medskip\noindent$(\Leftarrow)$
Suppose that $\phi$ has the five properties above. The definition
$\Phi(\vec{x}) = \phi(|\vec{x}|\!\!\downarrow)$ ensures that $\Phi$ is absolute
and permutation invariant.  We first verify monotonicity.  Suppose that $| \vec{x} | \leq_{\e} | \vec{y}|$.  Let $J$ denote the set of indices corresponding to the $k$ largest entries of $|\vec{x}|$.  Then
\begin{equation*}
    \sum_{j=1}^k | \vec{x} |_{[j]} = \sum_{j \in J} |x_j|
    \leq \sum_{j \in J} |y_j| \leq \sum_{j=1}^k |\vec{y}|_{[j]}.
\end{equation*}
Thus, $|\vec{x}|\!\!\downarrow \prec_{\w} | \vec{y} | \!\! \downarrow$, so (e) gives
\begin{equation*}
    \Phi(\vec{x}) = \phi( | \vec{x} | \!\! \downarrow) 
    \leq \phi( | \vec{y} | \!\! \downarrow) = \Phi( \vec{y}).
\end{equation*}
Therefore, $\Phi$ is monotone.

For all $\vec{x} \in \R^n$, (a) implies that
\begin{equation*}
\Phi(\vec{x}) = \phi(|\vec{x}|\!\!\downarrow) \geq 0
\end{equation*}
and (b) ensures that
\begin{equation*}
\Phi(\vec{x}) = 0 
\quad\iff\quad
\phi(|\vec{x}|\!\!\downarrow) =0 
\quad\iff\quad
|\vec{x}|\!\!\downarrow \,\, =0
\quad\iff\quad
\vec{x} =0.
\end{equation*}
By (c), for each $a \in \R$ and each $\vec{x} \in \R^n$, we have
\begin{equation*}
\Phi(a\vec{x}) = \phi(|a\vec{x}|\!\!\downarrow) = \phi(|a| \, |\vec{x}|\!\!\downarrow) = |a| \phi(|\vec{x}|\!\!\downarrow) = |a| \Phi(\vec{x}).
\end{equation*}
Now assume that $\vec{x}, \vec{y} \in \R^n$ and note that
\begin{equation}\label{E:UIN-Further-examples}
|\vec{x}+\vec{y}|\!\!\downarrow\,\, \prec_{\w} |\vec{x}|\!\!\downarrow+|\vec{y}|\!\!\downarrow.  
\end{equation}
Indeed, suppose that $\sigma,\tau \in \mathfrak{S}_n$ satisfy
\begin{equation*}
|x_{\sigma(1)}| \geq |x_{\sigma(2)}| \geq \cdots \geq |x_{\sigma(n)}|,
\quad\text{and}\quad
|y_{\tau(1)}| \geq |y_{\tau(2)}| \geq \cdots \geq |y_{\tau(n)}|.
\end{equation*}
Then, for each $1 \leq k \leq n$ and for any $\rho \in \mathfrak{S}_n$, 
\begin{align*}
\sum_{j=1}^{k} |x_{\rho(j)}+y_{\rho(j)}| &\leq \sum_{j=1}^{k} |x_{\rho(j)}| + \sum_{j=1}^{k} |y_{\rho(j)}|\\
&\leq \sum_{j=1}^{k} |x_{\sigma(j)}| + \sum_{j=1}^{k} |y_{\tau(j)}|,
\end{align*}
which ensures the validity of \eqref{E:UIN-Further-examples}. Therefore, (d), (e), and \eqref{E:UIN-Further-examples} imply that for any $\vec{x},\vec{y} \in \R^n$, 
\begin{align*}
\Phi(\vec{x}+\vec{y}) &= \phi(|\vec{x}+\vec{y}|\!\!\downarrow) \\
&\leq \phi(|\vec{x}|\!\!\downarrow+|\vec{y}|\!\!\downarrow)\\
&\leq \phi(|\vec{x}|\!\!\downarrow) + \phi\big(|\vec{y}|\!\!\downarrow)\\
&= \Phi(\vec{x}) + \Phi(\vec{y}).
\end{align*}
Therefore, $\Phi$ is a symmetric gauge function on $\R^n$.
\end{proof}

Here are some examples of unitarily invariant norms on $\M_n$.

\begin{example}[The Schatten $p$-Norm] 
Let $1 \leq p < \infty$ and define
\begin{equation*}
\Phi_p(\vec{x}) := \bigg( \sum_{k=1}^{n} |x_k|^p \bigg)^{1/p}
\end{equation*}
for all $\vec{x} \in \R^n$.
For $p=\infty$, define
\begin{equation*}
\Phi_\infty(\vec{x}) := \max\{ |x_k| : 1 \leq k \leq n\}.
\end{equation*}
One can verify that each of these is a symmetric gauge function on $\R^n$.  
Thus, von Neumann's theorem (Theorem \ref{T:UIN-von-Neumann}) implies that
\begin{equation}\label{E:UIN-spectral-norm-s22}
\|A\|_{\mathcal{S}_p} := \left( \sum_{k=1}^{n} |s_k(A)|^p \right)^{1/p},
\end{equation}
for $1 \leq p < \infty$, and
\begin{equation}\label{E:UIN-spectral-norm-s11}
\|A\|_{\mathcal{S}_\infty} := \max\{ |s_k(A)| : 1 \leq k \leq n\} = s_1(A),
\end{equation}
are unitarily invariant norms on $\M_n$. These are the \emph{Schatten $p$-norms}.
\end{example}

\begin{example}[The Ky Fan $k$-Norms] 
Fix $1 \leq k \leq n$ and define
\begin{equation*}
\Phi_{(k)}(\vec{x}) := \max\bigg\{ \sum_{j=1}^{k} |x_{\sigma(j)}| : \sigma \in \mathfrak{S}_n \bigg\}
\end{equation*}
for $\vec{x} \in \R^n$.  Equivalently, we could define
\begin{equation*}
\Phi_{(k)}(\vec{x}) := \max\bigg\{ \sum_{j=1}^{k} |x_{i_j}| : 1 \leq i_1 < i_2 < \cdots < i_k \leq n \bigg\}.
\end{equation*}
In this case, it is easier to use Lemma \ref{L:UIN-Further-examples} to show that $\Phi_{(k)}$ is a symmetric gauge function on $\R^n$. In fact, for each $\vec{x} \in \R^n_{+}\!\!\downarrow$,
\begin{equation}\label{E:UIN-def-phik}
\phi_{(k)}(\vec{x}) = \sum_{j=1}^{k} x_{j},
\end{equation}
so all five properties of $\phi$ are immediate. Thus, von Neumann's theorem \ref{T:UIN-von-Neumann}
ensures that for each $1 \leq k \leq n$
\begin{equation}\label{E:UIN-spectral-norm-s111}
\|A\|_{(k)} := \sum_{j=1}^{k} s_{j}(A)
\end{equation}
is a unitarily invariant norm on $\M_n$. These are the \emph{Ky Fan $k$-norms}.
\end{example}

By \eqref{E:UIN-spectral-norm-s1}, \eqref{E:UIN-spectral-norm-s11}, and \eqref{E:UIN-spectral-norm-s111}, we see that $\|\cdot\|_{\mathcal{S}_\infty}$ and $\|\cdot\|_{(1)}$ coincide with the spectral norm $\|\cdot\|_{\mathrm{op}}$. Similarly, \eqref{E:UIN-spectral-norm-s2} and \eqref{E:UIN-spectral-norm-s22} ensure that $\|\cdot\|_{\mathcal{S}_2}$ coincides with the Frobenius norm $\|\cdot\|_{\F}$. Also note that $\|\cdot\|_{\mathcal{S}_1} = \|\cdot\|_{(n)}$, which is the \emph{trace norm}. To distinguish it further, we denote it by $\|\cdot\|_{\mathcal{T}}$.
We summarize this discussion below:
\begin{align*}
\|\cdot\|_{\F}\  &= \|\cdot\|_{\mathcal{S}_2}\\
\|\cdot\|_{\mathrm{op}} &= \|\cdot\|_{\mathcal{S}_\infty} = \|\cdot\|_{(1)}\\
\|\cdot\|_{\mathcal{T}}\  &= \|\cdot\|_{\mathcal{S}_1} = \|\cdot\|_{(n)}.
\end{align*}

\begin{example}[The $\gamma$-Norm] 
Let $\boldsymbol{\gamma} = (\gamma_1,\gamma_2,\ldots,\gamma_n) \in \R^n_{+}\!\!\downarrow$ with $\gamma_1>0$.
For all $\vec{x} \in \R^n_{+}\!\!\downarrow$, define
\begin{equation*}
\phi_{\boldsymbol{\gamma}}(\vec{x}) := \sum_{j=1}^{n} \gamma_j x_j.
\end{equation*}
We can verify directly that $\phi_{\boldsymbol{\gamma}}$ satisfies the conditions of Lemma \ref{L:UIN-Further-examples} or use summation by parts to get
\begin{equation}\label{E:UIN-spectral-norm-s111-ga}
\phi_{\boldsymbol{\gamma}} = \sum_{j=1}^{n-1} (\gamma_j-\gamma_{j+1}) \phi_{(j)} + \gamma_n \phi_{(n)},
\end{equation}
in which $\phi_{(j)}$ is given by \eqref{E:UIN-def-phik}. The function $\phi_{\boldsymbol{\gamma}}$ is a nonnegative linear combination of functions $\phi_{(j)}$ satisfying the conditions of Lemma \ref{L:UIN-Further-examples}.  This representation gives (a), (c), (d), and (e); since $\gamma_1>0$, we see that (b) also holds. Thus, von Neumann's theorem \ref{T:UIN-von-Neumann} ensures that
\begin{equation}\label{E:UIN-spectral-norm-s111-g}
\|A\|_{\gamma} := \sum_{j=1}^{n} \gamma_j s_{j}(A)
\end{equation}
is a unitarily invariant norm on $\M_n$.    
\end{example}


\section{The Ky Fan domination principle} \label{S:UIN-further-applications}
In what follows, we employ the superscript ${}^{\diamond}$ to denote duality, as opposed to the more common ${}^*$.  This avoids confusion with the conjugate transpose of a matrix, which appears more often in this survey.

Although a dual norm can be defined for any normed vector space, we restrict ourselves to $\R^n$ and $\C^n$. We present our results for $\C^n$; they adapt readily to $\R^n$. Let $\|\cdot\|$ be a norm on $\C^n$. Fix $\vec{y} \in \C^n$, and consider the linear functional $\Lambda_{\vec{y}}\in (\C^n)^{\diamond}$ defined by
\begin{equation*}
\begin{array}{cccc}
\Lambda_{\vec{y}}: & (\C^n,\|\cdot\|) & \longrightarrow & \C\\
& \vec{x} & \longmapsto & \langle \vec{x},\vec{y} \rangle_{\C^n}.
\end{array}
\end{equation*}
The \emph{dual norm} of $\vec{y} \in \C^n$ is the operator norm of $\Lambda_{\vec{y}}$, denoted $\|\vec{y}\|^{\diamond}$:
\begin{equation}\label{E:UIN-def-norm-dual1}
\|\vec{y}\|^{\diamond} = \sup_{\substack{ \vec{x} \in \C^n \\ \vec{x} \neq \vec{0} }} \frac{|\langle \vec{x},\vec{y} \rangle|}{\|\vec{x}\|}.
\end{equation}
The dual norm is often denoted by $\|\vec{y}\|^{*}$, but to avoid confusion with the conjugate transpose of a matrix, we have adopted the diamond superscript. 

For all $\vec{x},\vec{y} \in \C^n$, the definition \eqref{E:UIN-def-norm-dual1} ensures that
\begin{equation}\label{E:UIN-def-norm-dual2}
|\langle \vec{x},\vec{y} \rangle| \leq \|\vec{x}\| \, \|\vec{y}\|^{\diamond}.
\end{equation}
The Hahn–-Banach theorem, together with the Riesz representation theorem, ensures that for
each $\vec{x}\in \C^n$, there is a $\vec{y}_{\vec{x}} \in \C^n$ with $\|\vec{y}_{\vec{x}}\|^{\diamond}=1$ such that
\begin{equation}\label{E:UIN-def-norm-dual3}
\|\vec{x}\| = \langle \vec{x},\vec{y}_{\vec{x}} \rangle.
\end{equation}
If we consider $\vec{x}$ as an element of $(\C^n)^{\diamond\diamond}$, then by \eqref{E:UIN-def-norm-dual2}, 
\begin{equation}\label{E:UIN-def-norm-dual5}
\|\vec{x}\|^{\diamond\diamond} = \sup_{\substack{\vec{y} \in \C^n \\ \vec{y} \neq \vec{0}}} \frac{|\langle \vec{y},\vec{x} \rangle|}{\|\vec{y}\|^{\diamond}} \leq \|\vec{x}\|.
\end{equation}
However, \eqref{E:UIN-def-norm-dual3} says that
\begin{equation*}
\|\vec{x}\|^{\diamond\diamond} 
= \sup_{\substack{\vec{y} \in \C^n \\ \vec{y} \neq \vec{0}} }\frac{|\langle \vec{y},\vec{x} \rangle|}{\|\vec{y}\|^{\diamond}} 
\geq \langle \vec{y}_{\vec{x}},\vec{x} \rangle =
\|\vec{x}\|.
\end{equation*}
For all $\vec{x} \in \C^n$, it follows that
\begin{equation}\label{E:UIN-def-norm-dual4}
\|\vec{x}\|^{\diamond\diamond} = \|\vec{x}\|
\end{equation}
and, by \eqref{E:UIN-def-norm-dual5},
\begin{equation}\label{E:UIN-def-norm-dual6}
\|\vec{x}\| = \sup_{\substack{\vec{y} \in \C^n \\ \vec{y} \neq \vec{0}}} \frac{|\langle \vec{y},\vec{x} \rangle|}{\|\vec{y}\|^{\diamond}}.
\end{equation}
This identity is the dual counterpart of \eqref{E:UIN-def-norm-dual1}.

Consider $\M_n$ as the vector space $\C^{n^2}$. A short calculation confirms that the Euclidean inner product on $\C^{n^2}$ equals the Frobenius inner product arising from the norm \eqref{E:UIN-def-norm-F} in $\M_n$; that is,
\begin{equation}\label{E:UIN-inner-prod-trace}
\langle A,B \rangle_{\F} = \tr (B^*A)  
\end{equation}
for all $A,B \in \M_n$.
If $\|\cdot\|$ is a unitarily invariant norm on $\M_n$, then, by \eqref{E:UIN-def-norm-dual1},
\begin{align*}
\|UAV\|^{\diamond} &= \sup_{\substack{B \in \M_n \\ B \neq 0}} \frac{|\langle B,UAV \rangle_{\F}|}{\|B\|}\\
&= \sup_{\substack{B \in \M_n \\ B \neq 0}} \frac{|\langle U^*BV^*,A \rangle_{\F}|}{\|B\|} 
&& (\text{by trace properties})\\
&= \sup_{\substack{B \in \M_n \\ B \neq 0}} \frac{|\langle U^*BV^*,A \rangle_{\F}|}{\|U^*BV^*\|} && (\text{since $\|\cdot\|$ is unitarily invariant})\\
&= \sup_{\substack{C \in \M_n \\ C \neq 0}} \frac{|\langle C,A \rangle_{\F}|}{\|C\|} 
&& (\text{change of variable})\\
&= \|A\|^{\diamond}.
\end{align*}
Therefore, the dual norm is also unitarily invariant.

\begin{lemma} 
If $\|\cdot\|$ is a unitarily invariant norm on $\M_n$, then for each $A \in \M_n$,
\begin{equation}\label{L:UIN-norm-rep-dual}
\|A\| =
\max \bigg\{
\sum_{j=1}^{n} s_j(A)s_j(B) : B \in \M_n, \, \|B\|^{\diamond}=1\bigg\}.
\end{equation}
\end{lemma}

\begin{proof}
First, apply \eqref{E:UIN-def-norm-dual6} to $\M_n = \C^{n^2}$, equipped with the Frobenius inner product \eqref{E:UIN-inner-prod-trace}, and get
\begin{equation}\label{E:UIN-norm-rep-dual-2}
\|A\| =
\max \big\{
|\tr (B^*A)| : B \in \M_n, \, \|B\|^{\diamond}=1\big\}.
\end{equation}
Then \eqref{C:UIN-s1=norm-7b} yields
\begin{equation*}
|\tr (B^*A)| \leq \sum_{j=1}^{n} s_j(B^*A),
\end{equation*}
and, by \eqref{C:UIN-s1=norm-5}, we get
\begin{equation*}
\sum_{j=1}^{n} s_j(B^*A) \leq \sum_{j=1}^{n} s_j(A) \, s_j(B).
\end{equation*}
Therefore, we have
\begin{equation*}
\|A\| \leq
\max \bigg\{
\sum_{j=1}^{n} s_j(A)s_j(B) : B \in \M_n, \, \|B\|^{\diamond}=1\bigg\}.
\end{equation*}

Next, fix any $B \in \M_n$ with $\|B\|^{\diamond}=1$. Since $\|\cdot\|^{\diamond}$ is unitarily invariant, we have
\begin{equation*}
\|B\|^{\diamond} = \|\diag \big(s(B)\big)\|^{\diamond}=1.
\end{equation*}
There are also unitaries $U, V \in \M_n$ such that
\begin{equation*}
A = U  \diag \big(s(A)\big)  V.
\end{equation*}
Let $C = U  \diag \big(s(B)\big)  V$. Since $\|\cdot\|^{\diamond}$ is unitarily invariant, we have
\begin{equation*}
\|C\|^{\diamond} = \|\diag \big(s(B)\big)\|^{\diamond}= \|B\|^{\diamond} = 1.
\end{equation*}
Thus, 
\begin{align*}
|\tr (C^*A)| &=|\tr (AC^*)| 
= |\tr \big(U  \diag \big(s(A)\big)  V  C^* \big)| \\
&= |\tr \big(U  \diag \big(s(A)\big)  \diag \big(s(B)\big)  U^* \big)| \\
&= |\tr \big(\diag \big(s(A)\big)  \diag \big(s(B)\big) \big)| \\
&= |\tr \big(\diag \big(s(A)*s(B)\big) \big)| \\
&= \sum_{j=1}^{n} s_j(A)s_j(B)
\end{align*}
and hence $\sum_{j=1}^{n} s_j(A)s_j(B) \leq \| A \|$.
Take the maximum over $B$ to get the reverse inequality.
\end{proof}

The norm representation \eqref{L:UIN-norm-rep-dual} suggests that we consider 
\begin{equation*}
\mathcal{E}_{\|\cdot\|} := \{ s(B) : B \in \M_n, \, \|B\|^{\diamond}=1 \} \subset \R^n_{+}\!\!\downarrow.
\end{equation*}
According to \eqref{E:UIN-spectral-norm-s111-g}, we can write \eqref{L:UIN-norm-rep-dual} as
\begin{equation}\label{E:UIN-norm-rep-dual-3}
\|A\| =
\max \{
\|A\|_{\gamma} :  \gamma \in \mathcal{E}_{\|\cdot\|}\}.
\end{equation}

Next is an important application of von Neumann's theorem.

\begin{theorem}[Ky Fan Domination Principle \cite{MR45952}] \label{T:Fan-dominion-P} 
For $A,B \in \M_n$, the following are equivalent.
\begin{enumerate}[(a)]\addtolength{\itemsep}{5pt}
\item $\|A\| \leq \|B\|$ for every unitarily invariant norm $\|\cdot\|$ on $\M_n$.

\item $\|A\|_{(k)} \leq \|B\|_{(k)}$ for all $1 \leq k \leq n$.

\item $\vec{s}(A) \prec_{\w} \vec{s}(B)$.
\end{enumerate}
\end{theorem}

\begin{proof}
\noindent (a) $\Rightarrow$ (b) This is immediate.

\medskip\noindent (b) $\Rightarrow$ (a) 
If $A= 0$, there is nothing to prove.
Suppose that $A \neq 0$. As in \eqref{E:UIN-spectral-norm-s111-ga}, summation by parts gives
\begin{equation*}
\|A\|_\gamma = \sum_{j=1}^{n-1} (\gamma_j-\gamma_{j+1}) \|A\|_{(j)} + \gamma_n \|A\|_{(n)}.
\end{equation*}
Therefore, assumption (b) implies
\begin{equation*}
\|A\|_{\gamma} \leq \|B\|_{\gamma}
\end{equation*}
for all $\boldsymbol{\gamma} = (\gamma_1,\gamma_2,\ldots,\gamma_n) \in \R^n_{+}\!\!\downarrow$ with $\gamma_1>0$; see \S\ref{S:UIN-further-examples}. 
If $\gamma \in \mathcal{E}_{\|\cdot\|}$, then $\boldsymbol{\gamma} = \vec{s}(C)$ for some
$C \in \M_n$ with $\| C \|^{\diamond} =1$.  In particular, $C \neq 0$ and hence
$\gamma_1 = s_1(C) > 0$.  Thus, the previous inequality applies to every
$\gamma \in \mathcal{E}_{\|\cdot\|}$.  Therefore, \eqref{E:UIN-norm-rep-dual-3}
ensures that $\|A \| \leq \|B\|$.

\medskip\noindent (b) $\Leftrightarrow$ (c) This follows from the definition \eqref{E:UIN-spectral-norm-s111} of the Ky Fan norms.
\end{proof}

The Ky Fan domination principle has several applications. We present two of them below. To effectively use Theorem \ref{T:Fan-dominion-P}, we establish an interesting result connecting $\|\cdot\|_{(k)}$ to $\|\cdot\|_{\mathcal{T}}$ and $\|\cdot\|_{\mathrm{op}}$. 

For any $X \in \M_n$ and any decomposition $X=Y+Z$, we have
\begin{align}
\|X\|_{(k)} &\leq \|Y\|_{(k)} + \|Z\|_{(k)} \notag\\
&\leq \|Y\|_{(n)} + k\|Z\|_{(1)} \notag\\
&= \|Y\|_{\mathcal{T}} + k\|Z\|_{\mathrm{op}}. \label{E:A-decmpoose-TS1}
\end{align}
These inequalities are sharp. If $X=U\diag (s_1,s_2,\ldots,s_n)V$ is the singular value decomposition of $X$, in which $U,V \in \M_n$ are unitary, then consider
\begin{equation*}
Y = U\diag (s_1-s_k,s_2-s_k,\ldots,s_k-s_k,0,\ldots,0)V
\end{equation*}
and
\begin{equation*}
Z = U\diag (s_k,s_k,\ldots,s_k,s_{k+1},\ldots,s_n)V.
\end{equation*}
Then, $X=Y+Z$ and
\begin{equation*}
\|Y\|_{\mathcal{T}} = (s_1-s_k)+(s_2-s_k)+\cdots+(s_k-s_k) = (s_1+s_2+\cdots+s_k)-ks_k
\end{equation*}
and
\begin{equation*}
\|Z\|_{\mathrm{op}} = s_k.
\end{equation*}
For this choice of $Y$ and $Z$, we have
\begin{equation}\label{E:A-decmpoose-TS1b}
\|X\|_{(k)} = \|Y\|_{\mathcal{T}} + k\|Z\|_{\mathrm{op}}.
\end{equation}

\begin{corollary}[Mirsky \cite{MR114821}] \label{C:mirsky-s}
For all $A,B \in \M_n$ and every unitarily invariant norm,
\begin{equation}\label{E:mirsky-s1}
\big\| \diag \big(s(A)-s(B)\big) \big\|
\leq
\|A-B\|.
\end{equation}
\end{corollary}

\begin{proof}
Here is the plan of the proof. In the first step, we prove \eqref{E:mirsky-s1} for the spectral norm $\|\cdot\|_{\mathrm{op}}$. Then we prove it for the trace norm $\|\cdot\|_{\mathcal{T}}$. In the next step, we use \eqref{E:A-decmpoose-TS1} and the previous two steps to extend the validity of \eqref{E:mirsky-s1} to the Ky Fan norms $\|\cdot\|_{(k)}$. Finally, we use the Ky Fan domination principle to establish the result for all unitarily invariant norms.

\medskip
\noindent \textsc{Step 1 (the spectral norm $\|\cdot\|_{\mathrm{op}}$):} By \eqref{E:UIN-max-min-formula-2}, we have
\begin{equation*}
s_i(A) = \min_{\substack{\V \subset \C^n \\ \dim \V=n-i+1}} \,\, \max_{\substack{\vec{x} \in \V \\ \|\vec{x}\|=1}} \,\, \|A\vec{x}\|.
\end{equation*}
If $X \in \M_n$ has rank at most $i-1$, then $\dim \ker(X) \geq n-i+1$. 
Choose an $(n-i+1)$-dimensional subspace $\V \subset \ker(X)$. This choice of $\V$ shows that
\begin{equation*}
s_i(A) \leq \max_{\substack{\vec{x} \in \V \\ \|\vec{x}\|=1}} \,\, \|Ax\| = \max_{\substack{\vec{x} \in \V \\ \|\vec{x}\|=1}} \,\, \|(A-X)x\| \leq \|A-X\|_{\mathrm{op}}.
\end{equation*}
Therefore, 
\begin{equation}\label{E:sk-leq-norm-A-X}
s_i(A) \leq \|A-X\|_{\mathrm{op}} \quad \text{whenever} \quad \rank(X) \leq i-1.
\end{equation}
Moreover, if $A=U\diag (s_1,s_2,\ldots,s_n)V$ is the singular value decomposition of $A$, consider
\begin{equation*}
X_0 = U\diag (s_1,s_2,\ldots,s_{i-1},0,0,\ldots,0)V.
\end{equation*}
Then $\rank(X_0) \leq i-1$ and also
\begin{equation*}
\|A-X_0\|_{\mathrm{op}} = \|U\diag (0,0,\ldots,0,s_{i},s_{i+1},\ldots,s_{n})V\|_{\mathrm{op}} = s_i.
\end{equation*}
Therefore, the inequality in \eqref{E:sk-leq-norm-A-X} is sharp in the sense that there is a matrix $X_0 \in \M_n$ with $\rank(X_0) \leq i-1$ such that
\begin{equation}\label{E:sk-leq-norm-A-X2}
s_i(A) = \|A-X_0\|_{\mathrm{op}}.
\end{equation}

Assume $X_0$ is as given in \eqref{E:sk-leq-norm-A-X2} for the matrix $A$. Then, for any other $B \in \M_n$, it follows from \eqref{E:sk-leq-norm-A-X} that
\begin{align*}
s_i(B) &\leq \|B-X_0\|_{\mathrm{op}}\\
&= \|(A-X_0)+(B-A)\|_{\mathrm{op}}\\
&\leq \|A-X_0\|_{\mathrm{op}}+\|B-A\|_{\mathrm{op}}\\
&\leq s_i(A)+\|B-A\|_{\mathrm{op}}.
\end{align*}
Changing the order of $A$ and $B$ above, we see that
\begin{equation*}
|s_i(A)-s_i(B)| \leq \|A-B\|_{\mathrm{op}} 
\end{equation*}
for all $1 \leq i \leq n$.  Thus,
\begin{equation}\label{E:sk-leq-norm-A-X3}
\big\| \diag \big(s(A)-s(B)\big) \big\|_{\mathrm{op}}
\leq
\|A-B\|_{\mathrm{op}}.
\end{equation}
In other words, \eqref{E:mirsky-s1} is valid for the spectral norm.

\medskip
\noindent \textsc{Step 2 (the trace norm $\|\cdot\|_{\mathcal{T}}$):}
We start with an observation about the trace norm. Let $X \in \h_n$ and let $X=U\diag (\lambda_1,\ldots,\lambda_n)U^*$ be its spectral decomposition. 
Let $p \in \{0,1,\ldots,n\}$ be such that
$\lambda_1 \geq \cdots \geq \lambda_p \geq 0$, while, if $p<n$, we have
$0>\lambda_{p+1} \geq \cdots \geq \lambda_n$.  Thus, $p=0$ means that all eigenvalues
are negative and $p=n$ means that all eigenvalues are nonnegative.
Define
\begin{equation*}
X_{+} = U\diag (\lambda_1,\ldots,\lambda_p,0,0,\ldots,0)U^*
\end{equation*}
and
\begin{equation*}
X_{-}= -U\diag (0,\ldots,0,\lambda_{p+1},\ldots,\lambda_n)U^*
\end{equation*}
and note that $X_{+} \geq 0$ and $X_{-} \geq 0$, so that $X=X_{+}-X_{-}$. 
This is the \emph{Jordan decomposition} of $X$. The property that we need is that
\begin{equation}\label{E:jordan-dec-trace-nprm}
\|X\|_{\mathcal{T}} = \tr (X_{+}) + \tr (X_{-}) = \sum_{j=1}^{n} |\lambda_j|.
\end{equation}

Suppose that $A,B \in \h_n$.  Apply \eqref{E:jordan-dec-trace-nprm} to $X=A-B$ and obtain
\begin{equation}\label{E:jordan-dec-trace-nprm2}
\|A-B\|_{\mathcal{T}} = \tr ((A-B)_{+}) + \tr ((A-B)_{-}).
\end{equation}
Let
\begin{equation*}
C = A+(A-B)_{-} = B+(A-B)_{+}
\end{equation*}
and observe that $C \geq A$ and $C \geq B$. 
The Weyl monotonicity principle (Theorem \ref{T:UIN-Cauchy-Weyl-1}) ensures that
\begin{equation*}
\lambda_j(C) \geq \lambda_j(A)
\qquad\text{and}\qquad
\lambda_j(C) \geq \lambda_j(B),
\end{equation*}
for $1 \leq j \leq n$. This yields
\begin{align*}
|\lambda_j(A)-\lambda_j(B)| &= |(\lambda_j(C)-\lambda_j(B)) - (\lambda_j(C)-\lambda_j(A))|\\
&\leq (\lambda_j(C)-\lambda_j(B)) + (\lambda_j(C)-\lambda_j(A) ) \\
&= \lambda_j(2C)-\lambda_j(A) -\lambda_j(B).
\end{align*}
Therefore, by \eqref{E:jordan-dec-trace-nprm2},
\begin{align}
\sum_{j=1}^{n} |\lambda_j(A)-\lambda_j(B)| &\leq \sum_{j=1}^{n}\lambda_j(2C)-\sum_{j=1}^{n}\lambda_j(A) -\sum_{j=1}^{n}\lambda_j(B) \notag\\
&= \tr (2C)-\tr (A)-\tr (B) \notag\\
&= \tr (C-A) + \tr (C-B) \notag\\
&= \tr ((A-B)_{+}) + \tr ((A-B)_{-}) \notag\\
&= \|A-B\|_{\mathcal{T}}. \label{E:diff-eigenvalue-a-b}
\end{align}
Next, we use the technique \eqref{E:def-gamma-x} and apply the estimate above to 
\begin{equation*}
\Gamma(A) =
\begin{bmatrix}
0 & A^* \\
A & 0
\end{bmatrix} \in \h_{2n}
\quad\text{and}\quad
\Gamma(B) =
\begin{bmatrix}
0 & B^* \\
B & 0
\end{bmatrix} \in \h_{2n}.
\end{equation*}
Recall that the first $n$ eigenvalues of $\Gamma(X)$ are the singular values of $X$ and the next $n$ eigenvalues are the negatives of these singular values in reverse order. Thus,  \eqref{E:diff-eigenvalue-a-b} becomes
\begin{align*}
2\sum_{j=1}^{n} |s_j(A)-s_j(B)|  &= \sum_{j=1}^{2n} |\lambda_j(\Gamma(A))-\lambda_j(\Gamma(B))|\\
&\leq \|\Gamma(A)-\Gamma(B)\|_{\mathcal{T}}\\
&= \|\Gamma(A-B)\|_{\mathcal{T}}\\
&= 2 \|A-B\|_{\mathcal{T}}
\end{align*}
and hence
\begin{equation}\label{E:sk-leq-norm-A-X4}
\big\| \diag \big(s(A)-s(B)\big) \big\|_{\mathcal{T}}
\leq
\|A-B\|_{\mathcal{T}}.
\end{equation}

\medskip
\noindent \textsc{Step 3 (the Ky Fan norms $\|\cdot\|_{(k)}$):} Fix $1 \leq k \leq n$ and apply \eqref{E:A-decmpoose-TS1b} to $X=A-B$. Then there exist $Y,Z\in \M_n$ such that $A-B=Y+Z$ and
\begin{equation}\label{Ea-b=y+kz}
\|A-B\|_{(k)} = \|Y\|_{\mathcal{T}} + k\|Z\|_{\mathrm{op}}.  
\end{equation}
Consider the trivial identity
\begin{equation*}
\diag \big(s(A)-s(B)\big)
=
\diag \big(s(Y+B)-s(B)\big)+
\diag \big(s(A)-s(Y+B)\big).
\end{equation*}
Observe that \eqref{E:A-decmpoose-TS1} implies that
\begin{align*}
\big\|\diag \big(s(A)-s(B)\big)\big\|_{(k)}
&\leq
\big\|\diag \big(s(Y+B)-s(B)\big)\big\|_{\mathcal{T}} \\
&+ k\big\|\diag \big(s(A)-s(Y+B)\big)\big\|_{\mathrm{op}}.
\end{align*}
Moreover, \eqref{E:sk-leq-norm-A-X4} implies that
\begin{equation*}
\big\|\diag \big(s(Y+B)-s(B)\big)\big\|_{\mathcal{T}} \leq \|(Y+B)-B\|_{\mathcal{T}} = \|Y\|_{\mathcal{T}}.
\end{equation*}
By \eqref{E:sk-leq-norm-A-X3},
\begin{equation*}
\big\|\diag \big(s(A)-s(Y+B)\big)\big\|_{\mathrm{op}} \leq \|A-(Y+B)\|_{\mathrm{op}} = \|Z\|_{\mathrm{op}}.
\end{equation*}
Therefore, by \eqref{E:sk-leq-norm-A-X3}, \eqref{E:sk-leq-norm-A-X4}, and \eqref{Ea-b=y+kz}, we have
\begin{equation*}
\big\|\diag \big(s(A)-s(B)\big)\big\|_{(k)} \leq 
\|Y\|_{\mathcal{T}} + k\|Z\|_{\mathrm{op}} = \|A-B\|_{(k)}.
\end{equation*}

\medskip
\noindent \textsc{Step 4 (general unitarily invariant norms):} Since we established the results for the Ky Fan $k$-norms, the general result is a consequence of the Ky Fan domination principle (Theorem \ref{T:Fan-dominion-P}).
\end{proof}

Note that the relation \eqref{E:mirsky-s1} is equivalent to 
\begin{equation}\label{E:mirsky-s2}
|s(A)-s(B)| \prec_{\w} s(A-B).
\end{equation}

\begin{corollary}[Lidskii \cite{MR39686}] \label{C:lidskii-s}
If $A,B \in \h_n$, then
\begin{equation*}
\lambda(A)-\lambda(B) \prec \lambda(A-B).
\end{equation*}
\end{corollary}

\begin{proof}
Recall that if $X$ is positive semidefinite, then $\boldsymbol{\lambda}(X)=\vec{s}(X)$. 
Since $A,B \in \h_n$, there are $a,b \in \R$ such that
\begin{equation*}
A+aI \geq B+bI \geq 0.
\end{equation*}
Indeed, choose $b$ so that $B+bI\geq 0$, then choose $a-b$ so large that $A - B + (a-b) I \geq 0$.
Hence,
\begin{equation*}
\vec{s}(A+aI) = \boldsymbol{\lambda}(A)+a\vec{1}
\quad\text{and}\quad
\vec{s}(B+bI) = \boldsymbol{\lambda}(B)+b\vec{1},
\end{equation*}
and
\begin{equation*}
\vec{s}(A-B+(a-b)I) = \boldsymbol{\lambda}(A-B)+(a-b)\vec{1}.
\end{equation*}
Then, applying \eqref{E:mirsky-s2} to $A+aI$ and $B+bI$, we obtain
\begin{equation*}
|\boldsymbol{\lambda}(A)-\boldsymbol{\lambda}(B)+(a-b)\vec{1}| \prec_{\w} \boldsymbol{\lambda}(A-B)+(a-b)\vec{1}.
\end{equation*}
Since
\begin{equation*}
\boldsymbol{\lambda}(A)-\boldsymbol{\lambda}(B)+(a-b)\vec{1} \prec_{\w} |\boldsymbol{\lambda}(A)-\boldsymbol{\lambda}(B)+(a-b)\vec{1}|
\end{equation*}
is immediate, we conclude that
\begin{equation*}
\boldsymbol{\lambda}(A)-\boldsymbol{\lambda}(B)+(a-b)\vec{1} \prec_{\w} \boldsymbol{\lambda}(A-B)+(a-b)\vec{1},
\end{equation*}
which implies
\begin{equation*}
\boldsymbol{\lambda}(A)-\boldsymbol{\lambda}(B) \prec_{\w} \boldsymbol{\lambda}(A-B).
\end{equation*}
However, we also have
\begin{align*}
\sum_{j=1}^{n} \lambda_j(A)-\lambda_j(B) &= \tr (A) - \tr (B)\\
&= \tr (A-B)\\
&= \sum_{j=1}^{n} \lambda_j(A-B).
\end{align*}
Thus,
\begin{equation*}
\boldsymbol{\lambda}(A)-\boldsymbol{\lambda}(B) \prec \lambda(A-B). \qedhere
\end{equation*}
\end{proof}

\begin{acknowledgement}
SRG was supported by National Science Foundation grant DMS-2452084; JM
was supported by the NSERC Discovery Grant (Canada); and MP was supported by the Ministry of Science and
Higher Education of the Republic of Poland.  After the entire article was drafted by the authors,
ChatGPT was consulted for proofreading and editing purposes.
\end{acknowledgement}

\bibliographystyle{plain}
\bibliography{Norm-References}
\end{document}